\documentclass[12pt]{iopart}

\expandafter\let\csname equation*\endcsname=\relax

\expandafter\let\csname endequation*\endcsname=\relax

\usepackage[left=2.5cm, right=3cm, top=3cm, bottom=2.5cm]{geometry}
\usepackage[bottom,hang]{footmisc}
\usepackage{amssymb,amsthm,amsfonts}
\usepackage{amsmath}
\usepackage{color}
\usepackage{bm}
\usepackage{graphicx}
\usepackage{multirow}
\usepackage[font=small]{caption}
\usepackage[hidelinks]{hyperref}

\DeclareMathOperator*{\argmin}{arg\,min}

\newtheorem{thm}{Theorem}
\newtheorem{lemma}[thm]{Lemma}
\newtheorem{prop}[thm]{Proposition}

\begin{document}

\def \mcI {\mathcal{I}}
\def \mcN {\mathcal{N}}
\def \mcT {\mathcal{T}}
\def \mcJ {\mathcal{J}}

\def \bbR {{\mathbb R}}
\def \bbN {{\mathbb N}}
\def \bbS {{\mathbb S}}
\def \bbD {{\mathbb D}}

\def \bmA {{\bm A}}
\def \bmB {{\bm B}}
\def \bmD {{\bm D}}
\def \bmE {{\bm E}}
\def \bmG {{\bm G}}
\def \bmH {{\bm H}}
\def \bmI {{\bm I}}
\def \bmP {{\bm P}}
\def \bmQ {{\bm Q}}
\def \bmR {{\bm R}}
\def \bmS {{\bm S}}
\def \bmW {{\bm W}}
\def \bmLambda {{\bm\Lambda}}
\def \bmgamma {{\bm\gamma}}

\def \bmb {{\bm b}}
\def \bmc {{\bm c}}
\def \bmd {{\bm d}}
\def \bmf {{\bm f}}
\def \bff {{\bm f}}
\def \bmg {{\bm g}}
\def \bmh {{\bm h}}
\def \bmm {{\bm m}}
\def \bmq {{\bm q}}
\def \bmu {{\bm u}}
\def \bmv {{\bm v}}
\def \bmw {{\bm w}}
\def \bmx {{\bm x}}
\def \bmy {{\bm y}}
\def \bmz {{\bm z}}
\def \bfx {{\bm x}}

\def \bR {\bbR}
\def \bRd {\bbR^{d}}
\def \bbRM {\bbR^{M}}
\def \bbRm {\bbR^{m}}
\def \bbRmo {\bbR^{m_0}}
\def \bbRn {\bbR^{n}}
\def \bbRd {\bbR^{d}}
\def \bbRpd {\bbR_+^{d}}
\def \bbRpn {\bbR_+^{n}}
\def \bbRnn {\bbR^{n\times n}}
\def \bbRmm {\bbR^{m\times m}}
\def \bbRmomo {\bbR^{m_0\times m_0}}
\def \bbRmod {\bbR^{m_0\times d}}
\def \bbRdd {\bbR^{d\times d}}
\def \bbSn {\bbS^{n}}
\def \bbSpd {\bbS_+^{d}}
\def \bbSpn {\bbS_+^{n}}
\def \bbSpm {\bbS_+^{m}}
\def \bbSpmo {\bbS_+^{m_0}}
\def \bbSkn {\bbS_k^{n}}
\def \bbSLn {\bbS_L^{n}}
\def \bbDppn {\bbD_{++}^{n}}
\def \bbDppm {\bbD_{++}^{m}}
\def \bbDppd {\bbD_{++}^{d}}

\def \indi {\iota_{\bbR_+^d}}
\def \prox {\mathrm{prox}}
\def \env {\mathrm{env}}
\def \diag {\mathrm{diag}}
\def \TMC {\text{TMC}}
\def \EM {\text{EM}}
\def \IEM {\text{IEM}}

\def \xk {{\bm x}^{k}}
\def \yk {{\bm y}^{k}}
\def \zk {{\bm z}^{k}}
\def \fk {{\bm f}^{k}}
\def \hk {{\bm h}^{k}}
\def \bk {{\bm b}^{k}}
\def \ck {{\bm c}^{k}}
\def \pk {p^{k}}
\def \qk {{\bm q}^{k}}
\def \tk {t^{k}}
\def \uk {{\bm u}^{k}}
\def \vk {{\bm v}^{k}}

\def \xkk {{\bm x}^{k+1}}
\def \ykk {{\bm y}^{k+1}}
\def \zkk {{\bm z}^{k+1}}
\def \fkk {{\bm f}^{k+1}}
\def \hkk {{\bm h}^{k+1}}
\def \bkk {{\bm b}^{k+1}}
\def \ckk {{\bm c}^{k+1}}
\def \pkk {p^{k+1}}
\def \qkk {{\bm q}^{k+1}}
\def \tkk {t^{k+1}}
\def \ukk {{\bm u}^{k+1}}
\def \vkk {{\bm v}^{k+1}}
%\widehat{}
\def \tu {\tilde{{\bm u}}}
\def \tf {\tilde{{\bm f}}}
\def \tuk {\tilde{{\bm u}}^{k}}
\def \tvk {\tilde{{\bm v}}^{k}}
\def \txk {\tilde{{\bm x}}^{k}}
\def \tyk {\tilde{{\bm y}}^{k}}
\def \tzk {\tilde{{\bm z}}^{k}}
\def \tfk {\tilde{{\bm f}}^{k}}
\def \thk {\tilde{{\bm h}}^{k}}
\def \tbk {\tilde{{\bm b}}^{k}}
\def \tck {\tilde{{\bm c}}^{k}}
\def \tqk {\tilde{{\bm q}}^{k}}

\def \tbkk {\tilde{{\bm b}}^{k+1}}
\def \tckk {\tilde{{\bm c}}^{k+1}}
\def \tqkk {\tilde{{\bm q}}^{k+1}}
\def \tukk {\tilde{{\bm u}}^{k+1}}
\def \tvkk {\tilde{{\bm v}}^{k+1}}
\def \txkk {\tilde{{\bm x}}^{k+1}}
\def \tykk {\tilde{{\bm y}}^{k+1}}
\def \tfkk {\tilde{{\bm f}}^{k+1}}
\def \thkk {\tilde{{\bm h}}^{k+1}}

\def \fkl {{\bm f}^{k,l}}
\def \fkll {{\bm f}^{k,l-1}}
\def \tfkl {\tilde{{\bm f}}^{k,l}}
\def \tfkll {\tilde{{\bm f}}^{k,l-1}}
\def \hkl {{\bm h}^{k,l}}
\def \hkll {{\bm h}^{k,l-1}}
\def \bkl {{\bm b}^{k,l}}
\def \bkll {{\bm b}^{k,l-1}}
\def \tbkl {\tilde{{\bm b}}^{k,l}}
\def \tbkll {\tilde{{\bm b}}^{k,l-1}}
\def \ckl {{\bm c}^{k,l}}
\def \ckll {{\bm c}^{k,l-1}}
\def \tckl {\tilde{{\bm c}}^{k,l}}
\def \tckll {\tilde{{\bm c}}^{k,l-1}}
\def \Skl {{\bm S}^{k,l}}
\def \Skll {{\bm S}^{k,l-1}}

\def \tvarphi {\tilde{\varphi}}

\newcommand\leqs{\leqslant}
\newcommand\geqs{\geqslant}

\title[APPGA with a generalized Nesterov momentum for PET reconstruction]{An accelerated preconditioned proximal gradient algorithm with a generalized Nesterov momentum for PET image reconstruction}

\author{Yizun Lin$^1$, Yongxin He$^1$, C. Ross Schmidtlein$^2$ and Deren Han$^{3*}$}

\footnotetext{\hspace{-0.5em}Author to whom correspondence should be addressed.}

\address{
$^1\ $Department of Mathematics, Jinan University, Guangzhou 510632, China\\
$^2\ $Department of Medical Physics, Memorial Sloan Kettering Cancer Center, New York, NY 10065 USA\\
$^3\ $School of Mathematical Sciences, Beihang University, Beijing 100191, China
}
\ead{linyizun@jnu.edu.cn, wheyongxin@163.com, schmidtr@mskcc.org and handr@buaa.edu.cn}
\vspace{10pt}
\begin{indented}
\item[]\today
\end{indented}

\begin{abstract}
This paper presents an Accelerated Preconditioned Proximal Gradient Algorithm (APPGA) for effectively solving a class of Positron Emission Tomography (PET) image reconstruction models with differentiable regularizers. We establish the convergence of APPGA with the Generalized Nesterov (GN) momentum scheme, demonstrating its ability to converge to a minimizer of the objective function with rates of $o\left(1/k^{2\omega}\right)$ and $o\left(1/k^{\omega}\right)$ in terms of the function value and the distance between consecutive iterates, respectively, where $\omega\in(0,1]$ is the power parameter of the GN momentum. To achieve an efficient algorithm with high-order convergence rate for the higher-order isotropic total variation (ITV) regularized PET image reconstruction model, we replace the ITV term by its smoothed version and subsequently apply  APPGA to solve the smoothed model. Numerical results presented in this work indicate that as $\omega\in(0,1]$ increase, APPGA converges at a progressively faster rate. Furthermore, APPGA exhibits superior performance compared to the preconditioned proximal gradient algorithm and the preconditioned Krasnoselskii-Mann algorithm. The extension of the GN momentum technique for solving a more complex optimization model with multiple nondifferentiable terms is also discussed.\\

\end{abstract}

%\maketitle

\section{Introduction}\label{sec:intro}

Positron Emission Tomography (PET) serves as a vital medical imaging modality, enabling noninvasive estimation of the biodistribution of radio-labeled drugs emitting positrons in patients. The generation of these images typically involves a reconstruction process, where tomographic data undergoes iterative ``deblurring" to yield progressively clearer images. However, owing to limitations in data counts stemming from health concerns, fundamental physical processes, and patient tolerability, the reconstructed images are often excessively noisy, unless reinforced by additional constraints such as regularization techniques \cite{yu2002edge}. Furthermore, the image reconstruction process must be completed within a few minutes after data acquisition. Consequently, achieving high-quality image  within a clinically feasible time frame poses a significant challenge. In this study, we introduce a novel smooth convex penalty and propose the utilization of a preconditioning technique alongside a Generalized Nesterov (GN) momentum scheme \cite{lin2023convergence}, to accelerate the reconstruction process of high-quality PET images.

PET images are reconstructed using models of physics and associated statistics, leading to a negative log-likelihood functional that can be minimized to find the most likely activity distribution function given the data. However, the immense data volume and  large number of elements (voxels) to be estimated pose significant challenges, making the resulting system models computationally demanding. To accelerate  convergence, numerous methods have been devised to shorten the iterations (time) required to approach the minima of the negative log-likelihood function. In PET, this is commonly achieved through preconditioning technique \cite{krol2012preconditioned,lin2019krasnoselskii,schmidtlein2017relaxed}, or incremental update scheme \cite{ahn2003globally,guo2022fast,Li2006A,schmidtlein2017relaxed}, which involves solving disjoint ordered subsets of the data \cite{de2001fast,Hudson1994Accelerated,mumcuoglu1994fast}. Without regularization, noise is controlled by truncating the updates before fitting noise degrades the image, followed by applying a post-filter. As this is an iterative deblurring process, the recovery of resolution in the absence of denoising is inherently limited by the influence of noise fitting.

Total variation (TV) regularization models \cite{rudin1992nonlinear}, have gained immense popularity in mitigating noise in various PET studies. More recently, several algorithms have been proposed for TV-type regularized PET image reconstruction models, including the Preconditioned Alternating Projection Algorithm (PAPA) \cite{krol2012preconditioned,li2015effective}, the Alternating Direction Method of Multipliers (ADMM) \cite{boyd2010distributed,chun2014alternating} and the Preconditioned Krasnoselskii-Mann Algorithm (PKMA) \cite{lin2019krasnoselskii}. Specifically, it has been shown in \cite{lin2019krasnoselskii} that PKMA outperforms both PAPA and ADMM for the higher-order TV (HOTV) regularized PET image reconstruction model. Besides, findings in \cite{krol2012preconditioned} highlight PAPA's superior performance, surpassing earlier TV-type methods, such as the nested EM-TV algorithm \cite{sawatzky2008accurate}, the one-step-late method with TV regularization \cite{panin1999total}, and the preconditioned primal-dual algorithm \cite{pock2011diagonal,sidky2012convex}.

Though first-order TV regularization is efficient in eliminating noise while preserving the edges and intricate details of an image, it is susceptible to two notable artifacts, namely ``corner" and ``staircase" effects. Corner artifacts originate from the specific norm employed in the TV functional. The TV semi-norms can be broadly classified into two types, anisotropic TV (ATV) and isotropic TV (ITV), where the ITV semi-norm suppresses ``corner" artifact formation. These semi-norms can be viewed as the composition of a convex function (the $\ell_1$ norm for the ATV or the $\ell_2$ norm for the ITV) with a linear transformation. It is well-known that the ATV can be smoothed by the Huber function, yet there have been few methods proposed to smooth ITV. As shown in Figure \ref{ATVITV_ReconUniform}, the employment of ATV as a regularizer for PET imaging often results in smooth structures appearing squared-off, which fails to accurately reflect the general biodistributions of radio-labeled drugs. Conversely, ITV does not exhibit this issue. Therefore, ITV stands as the superior choice for PET image reconstruction. The development of a smoothed ITV functional is addressed in this work.

The other prominent TV artifact in PET, ``staircase" artifacts, arises from the image's piecewise smoothness. To mitigate these artifacts while retaining the favorable features of TV's, an HOTV regularization method was proposed and has been shown to be promising for PET image reconstruction \cite{li2015effective,schmidtlein2017relaxed}.  As shown in Figure \ref{FOTVHOTV_ReconBrain}, the unregularized reconstructed image exhibits substantial noise. Although first-order TV (FOTV) regularization efficiently mitigates noise, it concurrently introduces staircase artifacts. Conversely, HOTV regularization proficiently eliminates noise while effectively removing staircase artifacts arising from the application of FOTV regularization. Though PKMA is one of the most efficient existing algorithms for HOTV regularized PET image reconstruction, it may not be able to guarantee a high order convergence rate due to the nonsmoothness of the regularizer. To improve the convergence rate of the solving algorithm, we can refine the regularizer through smoothing and accelerate the algorithm by incorporating momentum technique.

The most well-known momentum scheme is Nesterov's momentum \cite{nesterov1983method}. Beck and Teboulle showed that the Iterative Shrinkage-Thresholding Algorithm (ISTA) exhibits an $O\left(1/k\right)$ convergence rate in terms of the function value (FV-convergence rate), and ISTA with Nesterov's momentum (FISTA) can elevates the FV-convergence rate to $O(1/k^2)$. However, the convergence of the iterative sequence generated by FISTA remains unclear in their work \cite{beck2009fast}. Subsequently, Chambolle and Dossal (CD) addressed this gap in \cite{chambolle2015convergence}, proving that the iterative sequence for FISTA with a refined set of momentum parameters (FISTA-CD) both converged and had an  $O(1/k^2)$ FV-convergence rate. Later, Attouch and Peypouquet demonstrated that FISTA-CD can actually achieve an $o\left(1/k^2\right)$ FV-convergence rate and an $o\left(1/k\right)$ convergence rate in terms of the distance between consecutive iterates \cite{attouch2016rate}. To offer a broader range of momentum algorithms capable of achieving diverse convergence rates in various scenarios, Lin et al. \cite{lin2023convergence} presented a Generalized Nesterov (GN) momentum framework that surpasses both the conventional Nesterov momentum and CD momentum schemes, given a specific selection of momentum parameters. However, the algorithm discussed in \cite{lin2023convergence} lacks the preconditioning technique, which holds significant importance in improved convergence in PET image reconstruction.

In this work, we propose an accelerated preconditioned proximal gradient algorithm (APPGA) with GN momentum scheme for
% the
smoothed higher-order ITV (SHOITV) regularized PET image reconstruction model. We prove that the propose APPGA with the GN momentum scheme can converge to a minimizer of the objective function, achieving convergence rates of $o\left(1/k^{2\omega}\right)$ and $o\left(1/k^{\omega}\right)$ in terms of the function value and the distance between consecutive iterates, respectively, where $\omega\in(0,1]$ is the power parameter of the GN momentum. Numerical results demonstrate that as $\omega\in(0,1]$ increases, APPGA outperforms PKMA and PPGA for the SHOITV regularized PET image reconstruction. Moreover, we highlight extensive applicability of the GN momentum in tackling the original nonsmooth HOITV regularized PET model, emphasizing its robustness and adaptability in comparison to the conventional Nesterov momentum.

The remainder of this paper is organized as follows: In Section 2, we first describe the SHOITV regularized PET image reconstruction model, and then develop the APPGA to solve this model. We analyze in Section 3 the convergence and convergence rate of the APPGA. Section 4 presents the numerical results for comparison of our proposed APPGA with the PKMA and PPGA, subsequently discussing the extension of the GN momentum. Section 5 offers a conclusion.

\vspace{0.5em}
\begin{figure}[htbp]
\centering
\begin{tabular}{ccccc}
\includegraphics[width=0.205\textwidth]{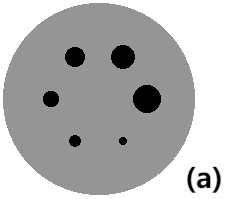}&
\includegraphics[width=0.2\textwidth]{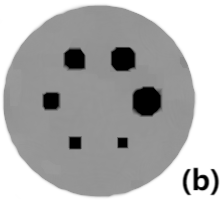}&
\includegraphics[width=0.2\textwidth]{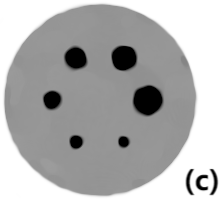}\\
\end{tabular}
\caption{(a) The original uniform phantom: uniform background with six uniform hot spheres of distinct radii; (b) the image reconstructed by using ATV regularization; (c) the image reconstructed by using ITV regularizatio.}
\label{ATVITV_ReconUniform}
\end{figure}
\vspace{-1em}

\begin{figure}[htbp]
\centering
\begin{tabular}{ccccc}
\includegraphics[width=0.18\textwidth]{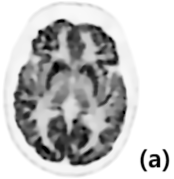}&
\includegraphics[width=0.18\textwidth]{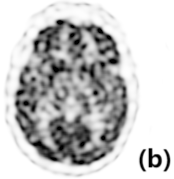}&
\includegraphics[width=0.18\textwidth]{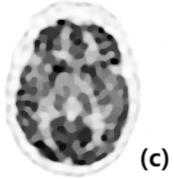}&
\includegraphics[width=0.18\textwidth]{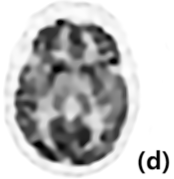}\\
\end{tabular}
\caption{(a) The original brain phantom: high quality clinical PET brain image; (b) the unregularized reconstructed image; (c) the image reconstructed by using first-order TV regularization; (d) the image reconstructed by using higher-order TV regularization.}
\label{FOTVHOTV_ReconBrain}
\end{figure}

\section{PET image reconstruction}\label{SecPETModel}
In this section, we introduce the Smoothed Higher-Order Isotropic Total Variation (SHOITV) regularization model for PET image reconstruction, and then develop an Accelerated Preconditioned Proximal Gradient Algorithm (APPGA) with a generalized Nesterov momentum scheme to solve the model.

\subsection{SHOITV regularized PET image reconstruction model}
We begin with describing the PET image reconstruction model with a HOITV regularizer. We denote by $\bbR_+$ the set of all nonnegative real numbers. For positive integers $m$ and $d$, we let $\boldsymbol{A} \in \mathbb{R}_+^{m \times d}$ denote the PET system matrix, whose $(i, j)$-th entry specifies the probability that the $i$th detector bin pair will be able to detect photon pairs released from the $j$th voxel of a radiotracer distribution $\bmf \in \mathbb{R}_+^d$ within a patient. Note that these probabilities are derived from a physical model of the PET system and object in the scanner. Let $\bmgamma \in \mathbb{R}_+^m$ denote the average background noise. The projection data $\bmg\in\mathbb{R}_+^m$ of the PET system can be calculated by
\begin{equation}\label{model-Poisson}
\bmg={\text{Poisson}}(\bmA\bmf+\bmgamma),
\end{equation}
where $\text{Poisson}(\bmx)$ denotes a Poisson-distributed random vector with mean $\bmx$. System \eqref{model-Poisson} may be solved by minimizing the fidelity term
\begin{equation}\label{FidelityTerm}
F(\bmf) := \langle\bmA\bmf,{\bf1}_m\rangle-\langle\ln\left(\bmA\bmf+\bmgamma\right),\bmg\rangle,
\end{equation}
where ${\bf1}_m\in\bbRm$ is the vector whose components are all 1, the logarithmic function at $\bmx\in\bbRn$ is defined by $\ln\bmx:=(\ln x_1, \ln x_2,\ldots,\ln x_n)^\top$, and $\langle \bmx,\bmy\rangle:=\sum_{i=1}^{n}x_iy_i$ is the inner product of $\bmx,\bmy\in\bbRn$.

It is well-known that system \eqref{model-Poisson} is ill-posed,
directly minimizing the fidelity term \eqref{FidelityTerm} will lead to severe noise in the reconstructed image \cite{yu2002edge}. To address this issue, regularization (penalty) terms were introduced as a part of the reconstruction model. Using both the first and second order Isotropic TV (ITV) penalties, we get the following HOITV regularized PET image reconstruction model
\begin{equation}\label{model:HOTV}
\argmin_{\bmf\in\bbRd}\left\{F(\bmf)+\varphi_1(\bmB_1\bmf)+\varphi_2(\bmB_2\bmf)+\iota(\bmf)\right\},
\end{equation}
where $\varphi_1\circ B_1$ and $\varphi_2\circ B_2$ represent the first and second-order ITV, respectively. The indicator function $\iota$ on $\bbRpd$, defined by
$$
\iota(\bmx):=\begin{cases}
0, & \text{if}\ \bmx\in\bbRpd,\\
+\infty, & \text{else},
\end{cases}
$$
is included in model \eqref{model:HOTV} to ensure the nonnegativity of $\bmf$.

We next recall the definition of HOITV.
For the sake of simplicity, we present only the 2D HOITV. For the 3D case, readers are referred to \cite{li2015effective}. We assume that $d=N^2$, for a positive integer $N$. We denote by $I_N$ the $N\times N$ identity matrix, $\bmD$ the $N\times N$ backward difference matrix with $D_{j,j}:=1$ and $D_{j,j-1}:=-1$ for $j=2,3,\ldots,N$, and all other entries zero. Through the matrix Kronecker product $\otimes$, the first-order difference matrix $\bmB_1\in\bbR^{2d\times d}$ and the second-order difference matrix $\bmB_2\in\bbR^{4d\times d}$ are defined, respectively, by
\begin{equation*}
\bmB_1:=\left[\begin{array}{c}
\bmI_N\otimes\bmD\\
\bmD\otimes\bmI_N
\end{array}\right],\
\bmB_2:=\left[\begin{array}{c}
\bmI_N\otimes(-\bmD^\top\bmD)\\
(-\bmD^\top)\otimes\bmD\\
(-\bmD^\top\bmD)\otimes\bmI_N\\
\bmD\otimes(-\bmD^\top)
\end{array}\right].
\end{equation*}
The definitions of $\varphi_1$ and $\varphi_2$ are given by
$$
\varphi_1(\bmx):=\lambda_1\sum_{i=1}^d\left\|(x_i,x_{d+i})^\top\right\|_2,\ \mbox{for}\ \bmx\in\bbR^{2d},
$$\vspace{-0.5em}
$$
\varphi_2(\bmx):=\lambda_2\sum_{i=1}^d\left\|(x_i,x_{d+i},x_{2d+i},x_{3d+i})^\top\right\|_2,\ \mbox{for}\ \bmx\in\bbR^{4d},
$$
where $\lambda_1,\lambda_2\in(0,+\infty)$ are the regularization parameters.

Note that the $\ell_2$ norm in HOITV is nondifferentiable, which is not conducive for the development of a solving algorithm with high-order convergence rate. To address this problem, we propose a smoothed approximation function $s_\epsilon$ of the $\ell_2$ norm, which is defined by
\begin{equation}\label{def:smoothL2}
s_\epsilon(\bmx)=\begin{cases}
\|\bmx\|_2-\frac{\epsilon}{2}, &\mbox{if}\ \|\bmx\|_2>\epsilon;\\
\frac{1}{2\epsilon}\|\bmx\|_2^2, &\mbox{otherwise},
\end{cases}
\end{equation}
where $\epsilon\in(0,+\infty)$. It is easy to see that when $\epsilon$ tends to zero, function $s_\epsilon$ tends to the $\ell_2$ norm $\|\cdot\|_2$. We shall show later in the proof of Lemma \ref{lem:phiconvLip} that $s_\epsilon$ is a convex function with a Lipschitz continuous gradient.

By replacing $\|\cdot\|_2$ with $s_\epsilon$ in the HOITV, the proposed Smoothed HOITV (SHOITV) regularized PET image reconstruction model can be written as
\begin{equation}\label{model:SHOTV}
\argmin_{\bmf\in\bbRd}\left\{F(\bmf)+\tvarphi_1(\bmB_1\bmf)+\tvarphi_2(\bmB_2\bmf)+\iota(\bmf)\right\},
\end{equation}
where $\tvarphi_1:\bbR^{2d}\to\bbR$ and $\tvarphi_2:\bbR^{4d}\to\bbR$ are defined by
\begin{gather*}
\tvarphi_1(\bmx):=\lambda_1\sum_{i=1}^d s_\epsilon\left((x_i,x_{d+i})^\top\right),\ \bmx\in\bbR^{2d},\\
\tvarphi_2(\bmx):=\lambda_2\sum_{i=1}^d s_\epsilon\left((x_i,x_{d+i},x_{2d+i},x_{3d+i})^\top\right),\ \bmx\in\bbR^{4d}.
\end{gather*}
For notational convenience, we define $\bmB:=\left(\begin{array}{c}
\bmB_1\\
\bmB_2
\end{array}\right)$ and
$$
\tvarphi(\bmx):=\tvarphi_1(\bmx_{1:2d})+\tvarphi_2(\bmx_{2d+1:6d}),\ \mbox{for}\ \bmx\in\bbR^{6d}.
$$
Then model \eqref{model:SHOTV} can be succinctly reformulated as the following two-term optimization model:
\begin{equation}\label{model:SHOTV2}
\argmin_{\bmf\in\bbRd}\left\{\phi(\bmf)+\iota(\bmf)\right\},
\end{equation}
where
\begin{equation}\label{def:phi}
\phi:=F+\tvarphi\circ\bmB
\end{equation}
is convex and differentiable with a Lipschitz continuous gradient (see Lemma \ref{lem:phiconvLip}).

% {\color{blue}Since $\phi$ is a functional shouldn't it be capitalized as $\Phi$ (or maybe $\Psi$ due to the ubiquitous use of $\phi$ in the paper)?}

\subsection{Accelerated preconditioned proximal gradient algorithm}\label{subsec:APPGA}
We then develop an Accelerated Preconditioned Proximal Gradient Algorithm (APPGA) for solving model \eqref{model:SHOTV2}. To this end, we characterize a solution of the model as a fixed-point of a mapping defined via the gradient descent operator of $\phi$ and the proximity operator of $\iota$. Throughout this paper, we let $\bbDppn$ denote the set of $n\times n$ diagonal matrices with positive diagonal entries, and $\Gamma_0(\bbRn)$ denote the class of all proper lower semicontinuous convex functions defined on $\bbRn$. For $\bmH\in\bbDppn$, the $\bmH$-weighted inner product is defined by $\langle\bmx,\bmy\rangle_\bmH:=\langle\bmx,\bmH\bmy\rangle$ for $\bmx,\bmy\in\bbRn$ and the corresponding $\bmH$-weighted $\ell_2$-norm is defined by $\|\bmx\|_\bmH:=\langle\bmx,\bmx\rangle_{\bmH}^{\frac{1}{2}}$. According to \cite{moreau1965proximite}, for a convex function $\psi:\bbRn\to\bbR$, the proximity operator of $\psi$ with respect to $\bmH\in\bbDppn$ at $\bmx\in\bbRn$ is defined by
\begin{equation}\label{def:prox}
\prox_{\psi,\bmH}(\bmx):=\argmin_{\bmu\in\bbRn}\left\{\tfrac{1}{2}\|\bmu-\bmx\|_{\bmH}^2+\psi(\bmu)\right\}.
\end{equation}
In particular, we use $\prox_{\psi}$ for $\prox_{\psi,\bmI}$. The subdifferential of $\psi$ at $\bmx$ is defined by
$$
\partial\psi(\bmx):=\{\bmy\in\bbRn:\psi(\bmz)\geqs\psi(\bmx)+\langle\bmy,\bmz-\bmx\rangle,\ \forall\bmz\in\bbRn\}.
$$

We now establish the fixed-point characterization of a solution of model \eqref{model:SHOTV2}. Throughout this paper, we denote by $L_{\phi}$ the Lipschitz constant of $\nabla\phi$, $\mcI$ the identity operator, and let operator $\mcT:\bbRd\to\bbRd$ be defined by
\begin{equation}\label{defmT}
\mcT:=\prox_{\iota,\bmP^{-1}}\circ\left(\mcI-\bmP\nabla\phi\right),
\end{equation}
where $\bmP\in\bbDppd$. For an operator $T:\bbRn\to\bbRn$, $\bmx\in\bbRn$ is said to be a fixed point of $T$ if $\bmx=T\bmx$.

\begin{prop}\label{prop:FPchar}
Let operator $\mcT$ be defined by
\eqref{defmT}, where $\bmP\in\bbDppd$. Then $\bmf^*$ is a solution of model \eqref{model:SHOTV2} if and only if it is a fixed point of $\mcT$.
\end{prop}
\begin{proof}
By Fermat's rule (Theorem 16.3 of \cite{bauschke2017convex}), we know that $\bmf^*$ is a solution of model \eqref{model:SHOTV2} if and only if ${\bm0}\in\nabla\phi(\bmf^*)+\partial\iota(\bmf^*)$, which is equivalent to
\begin{equation}\label{gradinsubiota}
-\bmP^{-1}\bmP\nabla\phi(\bmf^*)\in\partial\iota(\bmf^*).
\end{equation}
According to relationship between the proximity operator and the subdifferential in Proposition 2.6 of \cite{micchelli2011proximity}, we see that \eqref{gradinsubiota} is equivalent to
\begin{equation*}\label{FPcharinvP}
\bmf^*=\prox_{\iota,\bmP^{-1}}(\bmf^*-\bmP\nabla\phi(\bmf^*)),
\end{equation*}
which implies the desired result.
\end{proof}

The matrix $\bmP$ in \eqref{defmT} is introduced as a preconditioning matrix for accelerating the convergence speed of the proposed algorithm. According to the above fixed-point characterization, model \eqref{model:SHOTV2} can be solved via the fixed-point iteration of $\mcT$ given by
\begin{equation}\label{PPGAiter}
\fkk=\mcT\fk,
\end{equation}
which is called preconditioned proximal gradient algorithm (PPGA). To further accelerate the convergence and improve the convergence rate of the solving algorithm, we adopt the Generalized Nesterov (GN) momentum scheme \cite{lin2023convergence} for PPGA and obtain the following accelerated preconditioned proximal gradient algorithm (APPGA):
\begin{equation}\label{APPGAiter}
\begin{cases}
\tfk=\fk+\theta_k(\fk-\bmf^{k-1})\\
\fkk=\mcT\tfk
\end{cases},
\end{equation}
where $\left\{\theta_k\right\}_{k\in\bbN_+}$ are the momentum parameters. To perform APPGA, two initial vectors $\bmf^0,\bmf^1\in\bbRd$ are required to be given. The closed form of $\prox_{\iota,\bmP^{-1}}$ can be obtained from Equation (23) of \cite{lin2019krasnoselskii}. For any $\bmx\in\bbRd$,
\begin{equation}\label{cloformproiota}
\prox_{\iota,\bmP^{-1}}(\bmx)=\prox_{\iota}(\bmx)=\max(\bmx,\bm{0}).
\end{equation}
Then we can provide the complete iterative scheme of APPGA for solving model \eqref{model:SHOTV2} as follows:
$$
\begin{cases}
\tfk=\fk+\theta_k(\fk-\bmf^{k-1}),\\
\fkk=\max\Big(\tfk-\bmP\Big(\nabla F(\tfk)+\bmB^{\top}\nabla\tvarphi(\bmB\tfk)\Big),\bm{0}\Big),
\end{cases}
$$
where the momentum parameters $\left\{\theta_k\right\}_{k\in\bbN_+}$ are given by
\begin{equation}\label{GNtheta}
\theta_k=\frac{t_{k-1}-1}{t_k},\ t_{k-1}=a(k-1)^\omega+b,\ \ k\in\bbN_+.
\end{equation}

Let $\frac{\bmx}{\bmy}:=\left(\frac{x_1}{y_1},\frac{x_2}{y_2},\ldots,\frac{x_n}{y_n}\right)^\top$ for $\bmx,\bmy\in\bbRn$. We compute that the gradients of $F$ and $s_\epsilon$ are given by
\begin{equation*}\label{eq:gradFidelity}
\nabla F(\bmf)=\bmA^{\top}\left({\bf1}_m-\frac{\bmg}{\bmA\bmf+\bmgamma} \right)
\end{equation*}
and
\begin{equation}\label{gradenvL2}
\nabla s_\epsilon(\bmx)=\frac{\bmx}{\max\{\|\bmx\|_2,\epsilon\}},
\end{equation}
respectively. According to \eqref{gradenvL2}, the gradient of $\tvarphi$ at $\bmx\in\bbR^{6d}$ is given by
\begin{align*}
&\big(\nabla\tvarphi(\bmx)\big)_i=\frac{\lambda_1x_{i}}{\max\{\|(x_{i},x_{d+i})^\top\|_2,\epsilon\}},\\
&\big(\nabla\tvarphi(\bmx)\big)_{d+i}=\frac{\lambda_1x_{d+i}}{\max\{\|(x_{i},x_{d+i})^\top\|_2,\epsilon\}},\\
&\big(\nabla\tvarphi(\bmx)\big)_{jd+i}=\frac{\lambda_2x_{jd+i}}{\max\{\|(x_{2d+i},x_{3d+i},x_{4d+i},x_{5d+i})^\top\|_2,\epsilon\}},
\end{align*}
for $i=1,2,\ldots,d$ and $j=2,3,4,5$.

\section{Convergence analysis of APPGA}\label{sec:convAPPGA}
In this section, we analyze the convergence and convergence rate of APPGA. Throughout this section, we always let $\bmf^*$ be a global minimizer of $\Phi$, where $\Phi:=\phi+\iota$. We also let $\bmP\in\bbDppd$, $p_{\max}$ be the maximal diagonal entry of $\bmP$. For notational convenience, we define
\begin{equation}\label{defetaktauk}
\eta_k:=\Phi(\fk)-\Phi(\bmf^*),\ \ \tau_k:=\frac{1}{2}\left\|\fk-\bmf^{k-1}\right\|_{\bmP^{-1}}^2
\end{equation}
and
\begin{equation}\label{defhk}
\bmh^k:=t_k\tfk+(1-t_k)\fk,
\end{equation}
for $k\in\bbN_{+}$. Then we have the following proposition that is an essential tool for analyzing the convergence and convergence rate. We postpone the proof of this proposition until the desired Lemma \ref{lem:neqPhizPhiy} is established.

\begin{prop}\label{prop1forconverg}
Let $\theta_k = \frac{t_{k-1}-1}{t_k}$, where $t_k\neq0$ for all $k\in\bbN_+$, and define
\begin{equation}\label{defepsilon}
\varepsilon_k:=2t_{k-1}^2\eta_k+\left\|\bmh^k-\bmf^*\right\|_{\bmP^{-1}}^2.
\end{equation}
If $p_{\max}\leqs\frac{1}{L_{\phi}}$ and there exists $K\in\bbN_+$ such that $t_k(t_k-1)\leqs t_{k-1}^2$ for all $k>K$, then the following hold:
\begin{itemize}
\item[$(i)$] $\varepsilon_{k+1}\leqs \varepsilon_k$ for all $k > K$ and $\lim_{k\to\infty}\varepsilon_k$ exists,
\item[$(ii)$] $\eta_k\leqs\frac{\varepsilon_k}{2t_{k-1}^2}$ for all $k>K$.
\item[$(iii)$] $\sum_{k=1}^{\infty}\left[t_{k-1}^2-t_k(t_k-1)\right]\eta_k\leqs\frac{\varepsilon_1}{2}$.
\end{itemize}
\end{prop}

To prove Proposition \ref{prop1forconverg}, we first investigate the properties of $\phi$ in the following lemma.

\begin{lemma}\label{lem:phiconvLip}
Let $\phi:\bbRn\to\bbR$ be defined by \eqref{def:phi}. Then $\phi$ is convex and $\nabla\phi$ is Lipschitz continuous.
\end{lemma}
\begin{proof}
It follows from \cite{krol2012preconditioned} that the first term $F$ in $\phi$ is a convex function with a Lipschitz continuous gradient. To prove this lemma, it suffices to verify that $\tvarphi$ is a convex function with a Lipschitz continuous gradient, which is true as long as $s_{\epsilon}$ is a convex function with a Lipschitz continuous gradient.

We first prove the convexity of $s_{\epsilon}$. According to the first order condition for convexity (see Proposition B.3 in \cite{bertsekas1999nonlinprog}), it suffices to show that
\begin{equation}\label{neq:foconvseps}
s_{\epsilon}(\bmy)\geqs s_{\epsilon}(\bmx)+\langle\nabla s_{\epsilon}(\bmx),\bmy-\bmx\rangle
\end{equation}
holds for all $\bmx,\bmy\in\bbRn$. Note that function $\psi_1(\bmx):=\|\bmx\|_2-\frac{\epsilon}{2}$ and $\psi_2(\bmx):=\frac{1}{2\epsilon}\|\bmx\|_2^2$ are both convex on $\bbRn$. Let $\Omega:=\{\bmx\in\bbRn:\|\bmx\|_2\leqs\epsilon\}$. Then the first order condition for convexity together with the definition of $s_\epsilon$ in \eqref{def:smoothL2} implies that \eqref{neq:foconvseps} holds for all $\bmx,\bmy\in\Omega$ and all $\bmx,\bmy\in\bbRn\backslash\Omega$. For the case $\bmx\in\Omega$ and $\bmy\in\bbRn\backslash\Omega$, since
$$
\|\bmy\|_2-\frac{1}{\epsilon}\|\bmx\|_2\|\bmy\|_2\geqs\left(1-\frac{1}{\epsilon}\|\bmx\|_2\right)\epsilon=\epsilon-\|\bmx\|_2,
$$
we have
\begin{align*}
&\|\bmy\|_2-\frac{\epsilon}{2}-\frac{1}{2\epsilon}\|\bmx\|_2^2-\left\langle\frac{\bmx}{\epsilon},\bmy-\bmx\right\rangle\\
\geqs&\|\bmy\|_2-\frac{1}{\epsilon}\langle\bmx,\bmy\rangle+\frac{1}{2\epsilon}\|\bmx\|_2^2-
\frac{\epsilon}{2}\\
\geqs&\|\bmy\|_2-\frac{1}{\epsilon}\|\bmx\|_2\|\bmy\|_2+\frac{1}{2\epsilon}\|\bmx\|_2^2-
\frac{\epsilon}{2}\\
\geqs&\frac{\epsilon}{2}-\|\bmx\|_2+\frac{1}{2\epsilon}\|\bmx\|_2^2\\
\geqs&\frac{1}{2\epsilon}\left(\epsilon-\|\bmx\|_2\right)^2\geqs0,
\end{align*}
which implies that \eqref{neq:foconvseps} holds. Similarly, we can also prove that \eqref{neq:foconvseps} holds for the case $\bmx\in\bbRn\backslash\Omega$ and $\bmy\in\Omega$. Therefore, $s_\epsilon$ is a convex function.

We next prove the Lipschitz continuity of $\nabla s_{\epsilon}$. For $\bmx,\bmy\in\bbRn$, we let $a:=\max\{\|\bmx\|_2,\epsilon\}$, $b:=\max\{\|\bmy\|_2,\epsilon\}$. Then it is easy to verify that $|a-b|\leqs\|\bmx-\bmy\|_2$, which together with the facts $\frac{\|\bmy\|_2}{b}\leqs 1$ and $a\geqs\epsilon$, yields the relation
$$
\left\|\frac{\bmy}{a}-\frac{\bmy}{b}\right\|_2=\frac{|b-a|}{ab}\|\bmy\|_2\leqs\frac{|b-a|}{a}\leqs\frac{\|\bmx-\bmy\|_2}{\epsilon}.
$$
Now we have
\begin{align*}
\|\nabla s_\epsilon(\bmx)-\nabla s_\epsilon(\bmy)\|_2=&\left\|\frac{\bmx}{a}-\frac{\bmy}{b}\right\|_2\\
\leqs&\left\|\frac{\bmx}{a}-\frac{\bmy}{a}\right\|_2+\left\|\frac{\bmy}{a}-\frac{\bmy}{b}\right\|_2\\
\leqs&\frac{2}{\epsilon}\|\bmx-\bmy\|_2,
\end{align*}
which completes the proof.
\end{proof}

We then recall the well-known descent lemma (Proposition A.24 of \cite{bertsekas2009convex}), and use it together with Lemma \ref{lem:phiconvLip} to prove Lemma \ref{lem:neqPhizPhiy}.

\begin{lemma}[\bf{Descent lemma}]\label{lem:desclem}
Let $\psi:\bbRn\to\bbR$ be differentiable with an $L$-Lipschitz continuous gradient, where $L>0$. Then $\psi(\bmy)\leqs\psi(\bmx)+\langle\nabla\psi(\bmx),\bmy-\bmx\rangle+\frac{L}{2}\|\bmy-\bmx\|_2^{2}$, for all $\bmx,\bmy\in\bbRn$.
\end{lemma}

\begin{lemma}\label{lem:neqPhizPhiy}
For all $\bmx,\bmy\in\bbRd$, let $\bmz:=\mcT\bmy$. If $p_{\max}\leqs\frac{1}{L_{\phi}}$, then
$$
\Phi(\bmz)\leqs\Phi(\bmx)+\langle\bmy-\bmx,\bmy-\bmz\rangle_{\bmP^{-1}}-\frac{1}{2}\|\bmy-\bmz\|_{\bmP^{-1}}^2.
$$
\end{lemma}
\begin{proof}
Since $p_{\max}\leqs\frac{1}{L_{\phi}}$, we have $\|\bmz-\bmy\|_{\bmP^{-1}}^{2}\geqs L_{\phi}\|\bmz-\bmy\|_{2}^{2}$. Then it follows from Lemma \ref{lem:phiconvLip} and \ref{lem:desclem} that
\begin{equation}\label{neq:phizphiy}
\phi(\bmz)\leqs\phi(\bmy)+\langle\nabla\phi(\bmy),\bmz-\bmy\rangle+\frac{1}{2}\|\bmz-\bmy\|_{\bmP^{-1}}^{2}.
\end{equation}
Since $\bmz=\prox_{\iota,\bmP^{-1}}\left(\bmy-\bmP\nabla\phi(\bmy)\right)$, by Equation \eqref{cloformproiota}, we know that $\bmz\in\bbR_+^d$, and hence $\iota(\bmz)=0$. Then \eqref{neq:phizphiy} gives that
\begin{equation}\label{neq:Phizphix}
\Phi(\bmz)\leqs\phi(\bmy)+\langle\nabla\phi(\bmy),\bmz-\bmy\rangle+\frac{1}{2}\|\bmz-\bmy\|_{\bmP^{-1}}^{2}.
\end{equation}
In addition, $\bmz$ is a minimizer of function $\varphi$ defined by
$$
\varphi(\bmu):=\frac{1}{2}\|\bmu-(\bmy-\bmP\nabla\phi(\bmy))\|_{\bmP^{-1}}^{2}+\iota(\bmu),\ \ \bmu\in\bbRd.
$$
By Fermat's rule, there exists $\bmq\in\partial\iota(\bmz)$ such that
\begin{equation}\label{Pinvzxgradphi}
\bmP^{-1}(\bmz-\bmy)+\nabla\phi(\bmy)+\bmq=\bm{0}.
\end{equation}
Note that $\phi$, $\iota$ are both convex. We have
\begin{align*}
&\phi(\bmx)\geqs\phi(\bmy)+\langle\nabla\phi(\bmy),\bmx-\bmy\rangle,\\
&\iota(\bmx)\geqs\iota(\bmz)+\langle\bmx-\bmz,\bmq\rangle.
\end{align*}
Summing the above two inequalities yields
\begin{equation}\label{neq:Phiygeqphix}
\Phi(\bmx)\geqs\phi(\bmy)+\langle\nabla\phi(\bmy),\bmx-\bmy\rangle+\langle\bmx-\bmz,\bmq\rangle.
\end{equation}
Now by combining the two inequalities \eqref{neq:Phizphix} and \eqref{neq:Phiygeqphix}, and then using Equation \eqref{Pinvzxgradphi}, we obtain that
\begin{align*}
\Phi(\bmx)-\Phi(\bmz)&\geqs\langle\nabla\phi(\bmy),\bmx-\bmz\rangle+\langle\bmx-\bmz,\bmq\rangle-\frac{1}{2}\|\bmz-\bmy\|_{\bmP^{-1}}^{2}\\
&=\langle\bmx-\bmz,\nabla\phi(\bmy)+\bmq\rangle-\frac{1}{2}\|\bmz-\bmy\|_{\bmP^{-1}}^{2}\\
&=\langle\bmz-\bmx,\bmP^{-1}(\bmz-\bmy)\rangle-\frac{1}{2}\|\bmz-\bmy\|_{\bmP^{-1}}^{2}\\
&=\frac{1}{2}\|\bmz-\bmy\|_{\bmP^{-1}}^2+\langle\bmy-\bmx,\bmz-\bmy\rangle_{\bmP^{-1}},
\end{align*}
which implies the desired result.
\end{proof}

\begin{proof}[Proof of Proposition \ref{prop1forconverg}]
We first prove Item $(i)$. For $k\in\bbN_+$, by letting $\bmx=\fk$, $\bmy=\tfk$ and $\bmx=\bmf^*$, $\bmy=\tfk$, respectively, in Lemma \ref{lem:neqPhizPhiy}, and defining $\bmw^k:=\tfk-\fkk$, we have
\begin{equation}\label{neq:PhifkkPhifk}
\Phi(\fkk)\leqs\Phi(\fk)+\langle\tfk-\fk,\bmw^k\rangle_{\bmP^{-1}}-\frac{1}{2}\left\|\bmw^k\right\|_{\bmP^{-1}}^2,
\end{equation}
\begin{equation}\label{neq:PhifkkPhifstar}
\Phi(\fkk)\leqs\Phi(\bmf^*)+\langle\tfk-\bmf^*,\bmw^k\rangle_{\bmP^{-1}}-\frac{1}{2}\left\|\bmw^k\right\|_{\bmP^{-1}}^2.
\end{equation}
Let
$$
p_k:= 2t_k\left\langle\bmh^k-\bmf^*,\bmw^k\right\rangle_{\bmP^{-1}}-t_k^2\left\|\bmw^k\right\|_{\bmP^{-1}}^2,
$$
where $\bmh^k$ is defined by \eqref{defhk}. The combination $\left(1-\frac{1}{t_k}\right)\cdot$ \eqref{neq:PhifkkPhifk}+$\frac{1}{t_k} \cdot$\eqref{neq:PhifkkPhifstar} gives
$$
\Phi(\fkk)\leqs\left(1-\frac{1}{t_k}\right)\Phi(\fk)+\frac{1}{t_k}\Phi(\bmf^*)+\frac{1}{2t_k^2}p_k,
$$
that is,
\begin{equation}\label{neq:etakketakpk}
\eta_{k+1}\leqs\left(1-\frac{1}{t_k}\right)\eta_k+\frac{1}{2t_k^2}p_k,\ \ \mbox{for all}\ k\in\bbN_+.
\end{equation}

It is necessary to verify that the equality
\begin{equation}\label{eq:hkkhktfkfkk}
\bmh^{k+1}=\bmh^k-t_k\bmw^k,\ \ k\in\bbN_+
\end{equation}
holds true. Substituting $\tilde{\bmf}^{k+1}=\fkk+\theta_{k+1}(\fkk-\fk)$ into the definition of $\bmh^{k+1}$ in \eqref{defhk}, and then using the facts $t_{k+1}\theta_{k+1}=t_k-1$ and $(1-t_k)\fk=\bmh^k-t_k\tfk$, we get that
\begin{align}
\bmh^{k+1}&=t_{k+1}\left[ \fkk + \theta_{k+1}(\fkk-\fk)\right] + (1-t_{k+1})\fkk \nonumber\\
&=(1+t_{k+1}\theta_{k+1})\fkk-t_{k+1}\theta_{k+1}\fk\nonumber\\
\label{eq:hkkfkkfk}&=t_k\fkk+(1-t_k)\fk\\
&=t_k\fkk+\bmh^k-t_k\tfk,\nonumber
\end{align}
which implies \eqref{eq:hkkhktfkfkk}. Since
$$
p_k=\left\|\bmh^k-\bmf^*\right\|_{\bmP^{-1}}^2-\left\|(\bmh^k-\bmf^*)-t_k\bmw^k\right\|_{\bmP^{-1}}^2,
$$
it follows from \eqref{eq:hkkhktfkfkk} that
\begin{equation}\label{eq:pkhkfstar}
p_k=q_k-q_{k+1},
\end{equation}
where $q_k:=\left\|\bmh^k-\bmf^*\right\|_{\bmP^{-1}}^2$, $k\in\bbN_+$. Substituting \eqref{eq:pkhkfstar} into \eqref{neq:etakketakpk} yields
\begin{equation*}
\eta_{k+1}\leqs\left(1-\frac{1}{t_k}\right)\eta_k+\frac{1}{2 t_k^2}\left(q_k-q_{k+1}\right).
\end{equation*}
Multiplying both sides of the above inequality by $2t_k^2$ gives
\begin{align*}
2t_k^2 \eta_{k+1}&\leqs 2t_k(t_k-1)\eta_k+q_k-q_{k+1}\\
&=2t_{k-1}^2\eta_k+q_k-q_{k+1}-2\left[t_{k-1}^2-t_k(t_k-1)\right]\eta_k,
\end{align*}
that is,
\begin{equation}\label{neq:epskkepsk}
\varepsilon_{k+1}+2\left[t_{k-1}^2-t_k(t_k-1)\right]\eta_k\leqs\varepsilon_k.
\end{equation}
Since $\eta_k\geqs0$ and there exists $K\in\bbN_+$ such that $t_k(t_k-1)\leqs t_{k-1}^2$ for all $k>K$, Item $(i)$ follows from \eqref{neq:epskkepsk} immediately.

According to the definition of $\varepsilon_k$ and Item $(i)$, we obtain that $2 t_{k-1}^2\eta_k\leqs\varepsilon_k\leqs\varepsilon_K$ for all $k>K$, which implies Item $(ii)$. In addition, there exists $\varepsilon^*\geqs0$ such that $\lim_{k\to\infty}\varepsilon_k=\varepsilon^*$. Summing \eqref{neq:epskkepsk} for $k=1,2,\ldots,K'$ and letting $K'\to+\infty$ yield $2\sum_{k=1}^{\infty}\left[t_{k-1}^2-t_k(t_k-1)\right]\eta_k\leqs\varepsilon_1-\varepsilon^*\leqs\varepsilon_1$, which implies Item $(iii)$.
\end{proof}

To obtain the convergence and convergence rate results of APPGA, we need some hypotheses on the momentum parameters. We recall the Momentum-Condition proposed in \cite{lin2023convergence}. For a sequence $\{t_k\}_{k\in\bbN}\subset\bbR$, we say that it satisfies the {\bf Momentum-Condition} if the following hypotheses are satisfied:
\begin{itemize}
\item[$(i)$] $t_k\neq0$ for all $k\in\bbN_+$.
\item[$(ii)$] There exist $\rho\in\bbR_+$ and $K_1\in\bbN_+$ such that
\begin{equation}\label{neq:momencond2}
\hspace{-0.5em}1\leqs t_{k-1}<\rho\left[t_{k-1}^2-t_k(t_k-1)\right],\ \mbox{for all}\ k>K_1.
\end{equation}
\item[$(iii)$] There exist $c_1, c_2\in\bbR_+$ and $K_2\in\bbN_+$ such that
\begin{equation}\label{neq:momencond3}
c_1t_k\leqs t_{k-1}\leqs c_2t_k,\ \mbox{for all}\ k>K_2.
\end{equation}
\item[$(iv)$] $\lim_{k\to \infty}t_k=+\infty$ and $\sum_{k=1}^\infty\frac{1}{t_k}=+\infty$.
\end{itemize}

We now establish the boundedness of the two series $\sum_{k=1}^\infty t_{k-1}\eta_k$ and $\sum_{k=1}^\infty t_{k-1}\tau_k$.

\begin{prop}\label{prop:boundetaktauk}
Let $\theta_k=\frac{t_{k-1}-1}{t_k}$, $k\in\bbN_+$, where $\left\{t_k\right\}_{k\in\bbN}\subset\bbR$ satisfies Item $(i)-(iii)$ of Momentum-Condition. If $p_{\max}\leqs\frac{1}{L_{\phi}}$, then $\sum_{k=1}^\infty t_{k-1}\eta_k<+\infty$ and $\sum_{k=1}^\infty t_{k-1}\tau_k<+\infty$.
\end{prop}

\begin{proof}
Multiplying both sides of the second inequality of \eqref{neq:momencond2} by $\eta_k$ and summing the resulting inequality for $k$ from $K_1+1$ to infinity yields
$$
\sum_{k=K_1+1}^\infty t_{k-1}\eta_k\leqs\rho\sum_{k=K_1+1}^\infty \left[t_{k-1}^2-t_k(t_k-1)\right]\eta_k,
$$
which implies the boundedness of $\sum_{k=1}^\infty t_{k-1}\eta_k$ according to Item $(iii)$ in Proposition \ref{prop1forconverg}.

We then prove the boundedness of $\sum_{k=1}^\infty t_{k-1}\tau_k$. To this end, we show that
\begin{equation}\label{neq:etakplustauk}
\eta_{k+1}+\tau_{k+1}\leqs\eta_k+\theta_k^2\tau_k,\ \ \mbox{for all}\ k\in\bbN_+.
\end{equation}
From the proof of Proposition \ref{prop1forconverg}, we know that \eqref{neq:PhifkkPhifk} holds. Let $\bmd^k:=\theta_k(\fk-\bmf^{k-1})$, $k\in\bbN_+$. Then $\tfk=\fk+\bmd^k$. Substituting this into \eqref{neq:PhifkkPhifk} yields that
\begin{align*}
\Phi(\fkk)&\leqs\Phi(\fk)+\langle\bmd^k,\fk-\fkk+\bmd^k\rangle_{\bmP^{-1}}\\
&\hspace{1.2em}-\frac{1}{2}\left\|\fk-\fkk+\bmd^k\right\|_{\bmP^{-1}}^2\\
&=\Phi(\fk)-\frac{1}{2}\|\fk-\fkk\|_{\bmP^{-1}}^2-\frac{1}{2}\|\bmd^k\|_{\bmP^{-1}}^2,
\end{align*}
which implies \eqref{neq:etakplustauk}.

Multiplying  both sides of \eqref{neq:etakplustauk} by $t_k^2$ gives
$$
t_k^2(\eta_{k+1}+\tau_{k+1})\leqs t_k^2\eta_k+(t_{k-1}-1)^2\tau_k,
$$
that is,
\begin{equation}\label{neq:tauketak2}
\left(2t_{k-1}-1\right)\tau_k+t_k^2\tau_{k+1}-t_{k-1}^2\tau_k\leqs t_k^2 \left( \eta_k-\eta_{k+1}\right),
\end{equation}
for all $k\in\bbN_+$. Since $\left\{t_k\right\}_{k\in\bbN}$ satisfies Item $(ii)$ and $(iii)$ of Momentum-Condition, there exist $c>0$ and $K\in\bbN_+$ such that $t_{k-1}\geqs1$, $0<t_{k+1}\left(t_{k+1}-1\right)<t_k^2$ and $0<t_{k+1}\leqs ct_k$ for all $k>K$. These together with \eqref{neq:tauketak2} give that
\begin{align*}
t_{k-1}\tau_k+\left(t_k^2\tau_{k+1}-t_{k-1}^2 \tau_k\right)&\leqs t_k^2\eta_k-t_{k+1}\left(t_{k+1}-1\right)\eta_{k+1}\\
&\leqs\left(t_k^2\eta_k-t_{k+1}^2\eta_{k+1}\right)+ct_k\eta_{k+1},
\end{align*}
for all $k > K$. Summing the above inequality for $k=K+1,K+2,\ldots,M$, we obtain
\begin{align*}
\sum_{k=K+1}^M t_{k-1}\tau_k+t_M^2\tau_{M+1}-t_K^2\tau_{K+1}
\leqs t_{K+1}^2\eta_{K+1}-t_{M+1}^2\tau_{M+1}+c\sum_{k=K+1}^M t_k\eta_{k+1},
\end{align*}
which yields that
$$
\sum_{k=K+1}^M t_{k-1}\tau_k\leqs t_K^2\tau_{K+1}+t_{K+1}^2\eta_{K+1}+c\sum_{k=K+1}^M t_k\eta_{k+1}.
$$
Note that the boundedness of $\sum_{k=1}^\infty t_{k-1}\eta_k$ has already been proven. Now by letting $M\to+\infty$ in the above inequality, we find that $\sum_{k=K+1}^\infty t_{k-1}\tau_k<+\infty$, which implies the desired result.
\end{proof}

According to Proposition \ref{prop:boundetaktauk}, we are able to show that $\lim _{k\to\infty}t_{k-1}^2\left(\tau_k+\eta_k\right)=0$. Here we omit the proof for this result, since it is similar to the proof of Proposition 3.4 in \cite{lin2023convergence}.

\begin{prop}\label{prop:ttauetaktozero}
Let $\theta_k=\frac{t_{k-1}-1}{t_k}$, $k\in\bbN_+$, where $\left\{t_k\right\}_{k\in\bbN}\subset\bbR$ satisfies Momentum-Condition. If $p_{\max}\leqs\frac{1}{L_{\phi}}$, then $\lim _{k\to\infty}t_{k-1}^2\left(\tau_k+\eta_k\right)=0$.
\end{prop}

We also recall Proposition 3.5 in \cite{lin2023convergence} as the following lemma that shows a class of sequences satisfying Momentum-Condition.
\begin{lemma}\label{lem:exptkMomCond}
Let $t_k:=ak^{\omega}+b$, $k\in\bbN$, where $\omega\in(0,1]$, $a\in\bbR$ and $b\in\bbR\backslash\{-ak^{\omega}:k\in\bbN_+\}$. If either $\omega\in(0,1)$, $a\in(0,+\infty)$ or $\omega=1$, $a\in\left(0,\frac{1}{2}\right)$ holds, then $\{t_k\}_{k\in\bbN}$ satisfies Momentum-Condition.
\end{lemma}

To prove the convergence of APPGA, we still need the nonexpansiveness of operator $\mcT$. We say that $\mcT$ is nonexpansive with respect to $\bmH\in\bbDppd$ if $\|\mcT\bmx-\mcT\bmy\|_{\bmH}\leqs\|\bmx-\bmy\|_{\bmH}$ holds for all $\bmx,\bmy\in\bbRd$.
\begin{prop}\label{prop:Tnonexpans}
Let operator $\mcT$ be defined by
\eqref{defmT}, where $\bmP\in\bbDppd$. If $p_{\max}\leqs\frac{2}{L_{\phi}}$, then $\mcT$ is nonexpansive with respect to $\bmP^{-1}$.
\end{prop}
\begin{proof}
From the definition of nonexpansiveness, it is easy to see that the composition of two nonexpansive operators is still nonexpansive. To prove this proposition, it suffices to verify the nonexpansiveness (with respect to $\bmP^{-1}$) of operators $\prox_{\iota,\bmP^{-1}}$ and $\mcI-\bmP\nabla\phi$. It follows from Lemma 2.4 of \cite{combettes2005signal} that $\prox_{\iota,\bmP^{-1}}$ is firmly nonexpansive, and hence nonexpansive with respect to $\bmP^{-1}$ (see Remark 4.34 of \cite{bauschke2017convex}).
For any $\bmx,\bmy\in\bbRd$, we let $\bmz:=\nabla\phi(\bmx)-\nabla\phi(\bmy)$. The Baillon-Haddad theorem \cite{baillon1977quelques} gives that $\|\bmz\|_2^2\leqs L_{\phi}\langle\bmx-\bmy,\bmz\rangle$, which together with $p_{\max}\leqs\frac{2}{L_{\phi}}$ implies
$$
\|\bmP\|_2\|\bmz\|_2^2=p_{\max}\|\bmz\|_2^2\leqs2\langle\bmx-\bmy,\bmz\rangle.
$$
Then
\begin{align*}
&\|(\mcI-\bmP\nabla\phi)(\bmx)-(\mcI-\bmP\nabla\phi)(\bmy)\|_{\bmP^{-1}}^2\\
=&\|\bmx-\bmy\|_{\bmP^{-1}}^2+\|\bmP\bmz\|_{\bmP^{-1}}^2-2\langle\bmx-\bmy,\bmP\bmz\rangle_{\bmP^{-1}}\\
\leqs&\|\bmx-\bmy\|_{\bmP^{-1}}^2+\|\bmP\|_2\|\bmz\|_2^2-2\langle\bmx-\bmy,\bmz\rangle\\
\leqs&\|\bmx-\bmy\|_{\bmP^{-1}}^2.
\end{align*}
Hence $\mcI-\bmP\nabla\phi$ is nonexpansive expansive with respect to $\bmP^{-1}$, which completes the proof.
\end{proof}

We are now in a position to prove the main theorem of this section.
\begin{thm}\label{thm:mainconvrate}
Suppose that $\theta_k=\frac{t_{k-1}-1}{t_k}$, $k\in\bbN_{+}$, where $t_k:=ak^\omega+b$, $k\in\bbN$, $b\in\bbR\backslash\{-ak^{\omega}:k\in\bbN_+\}$. If $p_{\max}\leqs\frac{1}{L_{\phi}}$ and either $\omega\in\left(0,1\right)$, $a\in(0,+\infty)$ or $\omega=1$, $a\in\left(0,\frac{1}{2}\right)$ holds, then:
\begin{itemize}
\item[$(i)$] $\left\|\fk-\bmf^{k-1}\right\|_2=o\left(\frac{1}{k^{\omega}}\right)$,
\item[$(ii)$] $\Phi\left(\fk\right)-\Phi\left(\bmf^*\right)=o\left(\frac{1}{k^{2\omega}}\right)$,
\item[$(iii)$] $\left\{\fk\right\}_{k\in\bbN}$ converges to a minimizer of $\Phi$.
\end{itemize}
\end{thm}

\begin{proof}
We first prove Item $(i)$ and $(ii)$ together by employing Proposition \ref{prop:ttauetaktozero}. It follows from Lemma \ref{lem:exptkMomCond} that $\{t_k\}_{k\in\bbN}$ satisfies Momentum-Condition. Then we know from Proposition \ref{prop:ttauetaktozero} that
\begin{align}
\label{eq:limfkfkkzero}&\lim_{k\to\infty}t_{k-1}^2\left\|\fk-\bmf^{k-1}\right\|_{\bmP^{-1}}^2=0,\\
\label{eq:limPhifkfstarzero}&\lim_{k\to\infty}t_{k-1}^2\left(\Phi\left(\fk\right)-\Phi\left(\bmf^*\right)\right)=0.
\end{align}
Note that $\left\|\fk-\bmf^{k-1}\right\|_{\bmP^{-1}}^2\geqs \frac{1}{p_{\max}}\|\fk-\bmf^{k-1}\|_2^2$, which together with \eqref{eq:limfkfkkzero} and the fact $t_k=ak^{\omega}+b$ implies Item $(i)$. Similarly, Item $(ii)$ follows from \eqref{eq:limPhifkfstarzero} and $t_k=ak^{\omega}+b$ directly.

We then prove Item $(iii)$. To this end, we need to verify that $\lim_{k\to\infty}\|\mcT\bmf^k-\bmf^k\|_{\bmP^{-1}}=0$ and $\lim _{k\to\infty}\left\|\fk-\bmf^*\right\|_{\bmP^{-1}}$ exists for any fixed point $\bmf^*$ of $\mcT$. It is easy to see from Momentum-Condition that $\{\theta_k\}_{k\in\bbN_+}$ is bounded. In addition, it follows from \eqref{eq:limfkfkkzero} that $\lim_{k\to\infty}\|\bmf^{k}-\bmf^{k-1}\|_{\bmP^{-1}}=0$. These together with the nonexpansiveness of $\mcT$ (see Proposition \ref{prop:Tnonexpans}) yield that
\begin{align*}
\lim _{k\to\infty}\left\|\mcT\fk-\fkk\right\|_{\bmP^{-1}}&=\lim _{k\to\infty}\left\|\mcT\fk-\mcT\left(\fk+\theta_k\left(\fk-\bmf^{k-1}\right)\right)\right\|_{\bmP^{-1}}\\
&\leqs\lim_{k\to\infty}|\theta_k|\left\|\fk-\bmf^{k-1}\right\|_{\bmP^{-1}}=0.
\end{align*}
Hence
\begin{align*}
\lim_{k\to\infty}\left\|\mcT\fk-\fk\right\|_{\bmP^{-1}}\leqs\lim_{k\to\infty}\left\|\mcT\fk-\fkk\right\|_{\bmP^{-1}}+\lim_{k\to\infty}\left\|\fkk-\fk\right\|_{\bmP^{-1}}=0,
\end{align*}
which implies that $\lim_{k\to\infty}\left\|\mcT\fk-\fk\right\|_{\bmP^{-1}}=0$.

We next prove the existence of $\lim_{k\to\infty}\left\|\fk-\bmf^*\right\|_{\bmP^{-1}}$. To this end, we prove the existence of $\lim_{k\to\infty}\left\|\bmh^k-\bmf^*\right\|_{\bmP^{-1}}$, where $\left\{\bmh^k\right\}_{k \in \bbN_{+}}$ is defined by \eqref{defhk}. From Item $(i)$ of Proposition \ref{prop1forconverg} and Equation \eqref{eq:limPhifkfstarzero}, we see that $\lim_{k\to\infty}\varepsilon_k$ and $\lim_{k\to\infty}2t_{k-1}^2\eta_k$ exist, where $\varepsilon_k$ is defined by \eqref{defepsilon}. Then \eqref{defepsilon} implies the existence of $\lim_{k\to\infty}\left\|\bmh^k-\bmf^*\right\|_{\bmP^{-1}}$. From \eqref{eq:hkkfkkfk} in the proof of Proposition \ref{prop1forconverg}, we have
$\bmh^{k+1}=t_k\left(\fkk-\fk\right)+\fk$. Letting
$$
r_{k+1}:=|t_{k}|\left\|\fkk-\fk\right\|_{\bmP^{-1}},\ w_k:=\|\fk-\bmf^*\|_{\bmP^{-1}}
$$
for $k\in\bbN$ and using the triangle inequality, we have
$$
w_k-r_{k+1}\leqs\left\|\bmh^{k+1}-\bmf^*\right\|_{\bmP^{-1}}\leqs w_k+r_{k+1},
$$
that is,
\begin{equation}\label{neq:hkfstarrsk}
\left\|\bmh^{k+1}-\bmf^*\right\|_{\bmP^{-1}}-r_{k+1}\leqs w_k\leqs\left\|\bmh^{k+1}-\bmf^*\right\|_{\bmP^{-1}}+r_{k+1}.
\end{equation}
We also see from \eqref{eq:limfkfkkzero} that $\lim_{k\to\infty}r_k=0$. Now, \eqref{neq:hkfstarrsk} together with the existence of $\lim_{k\to\infty}\|\bmh^k-\bmf^*\|_{\bmP^{-1}}$ and the fact $\lim_{k\to\infty}r_k=0$ implies that $\lim_{k\to\infty}\|\fk-\bmf^*\|_{\bmP^{-1}}$ exists.

The existence of $\lim_{k\to\infty}\|\fk-\bmf^*\|_{\bmP^{-1}}$ implies the boundedness of sequence $\{\fk\}_{k\in\bbN}$. Then there exists a subsequence $\left\{\bmf^{k_j}\right\}_{j\in\bbN_{+}}$ of $\{\fk\}_{k\in\bbN_0}$ that converges to some vector $\hat{\bmf}$. We next prove that $\hat{\bmf}$ is a
fixed point of $\mcT$. By the nonexpansiveness of $\mcT$, we have that
$$
\lim_{j\to\infty}\|\mcT\hat{\bmf}-\mcT\bmf^{k_j}\|_{\bmP^{-1}}\leqs\lim_{j\to\infty}\|\hat{\bmf}-\bmf^{k_j}\|_{\bmP^{-1}}=0,
$$
which implies that $\mcT\hat{\bmf}=\mcT\bmf^{k_j}$. This together with the proved fact $\lim_{k\to\infty}\|\mcT\bmf^k-\bmf^k\|_{\bmP^{-1}}=0$ gives that
$$
\mcT\hat{\bmf}-\hat{\bmf}=\lim_{j\to\infty}(\mcT\bmf^{k_j}-\bmf^{k_j})=0,
$$
that is, $\hat{\bmf}$ is a fixed point of $\mcT$. Recall that $\lim _{k\to\infty}\|\fk-\bmf^*\|_{\bmP^{-1}}$ exists for any fixed point $\bmf^*$ of $\mcT$. Of course, $\lim _{k\to\infty}\|\fk-\hat{\bmf}\|_{\bmP^{-1}}$ exists, and hence
$$
\lim _{k\to\infty}\|\fk-\hat{\bmf}\|_{\bmP^{-1}}=\lim _{j\to\infty}\|\bmf^{k_j}-\hat{\bmf}\|_{\bmP^{-1}}=0,
$$
which together with Proposition \ref{prop:FPchar} implies Item $(iii)$ of this theorem.
\end{proof}

\section{Numerical results}
In this section, we first present several numerical results, which aim to evaluate the performance of APPGA for the SHOITV regularized PET image reconstruction model. The results provide insights into the reconstruction quality, convergence behavior, and computational efficiency of the algorithms. Next we demonstrate the broader applicability of the GN momentum to the original HOITV regularized PET image reconstruction model.

\subsection{Simulation setup for APPGA}
We implemented the algorithms via Matlab through a 2D PET simulation model as described in \cite{schmidtlein2017relaxed}. The number of
counts used in these 2D simulations were set to be equivalent to those from a 3D PET brain patient acquisition (administered 370 MBq FDG and imaged 1-hour post injection) collected from the central axial slice via GE D690/710 PET/CT. The resulting reference count distribution was used as the Poisson parameters for the noise realizations. An area integral projection method was used to build the projection matrix based on a cylindrical detector ring consisting of 576 detectors whose width are 4 mm. We set the FOV as 300 mm and use 288 projection angles to reconstruct a $256\times256$ image.

To simulate the physical factors that will affect the resolution of the reconstructed image, such as positron range, detector width, penetration, residual momentum of the positron and imperfect decoding, the phantom was convolved with an idealized (space-invariant, Gaussian) point spread function (PSF), which was set as a constant over the whole FOV. The full width half maximum (FWHM) of this PSF was set to 6.59 mm based on physical measurements from acceptance testing and \cite{moses2011fundamental}. The true count projection data was produced by forward-projecting the phantom convolved by the PSF. Uniform water attenuation (with the attenuation coefficient 0.096 cm$^{-1}$) was simulated using the PET image support. The background noise was implemented as describe in \cite{berthon2015petstep} and was based on 25\% scatter fraction and 25\% random fraction, given by $SF:=Sc/(Tc+Sc)$ and $RF:=Rc/(Tc+Sc+Rc)$, respectively, where $Tc$, $Sc$ and $Rc$ represent true, random, and scatter counts respectively. Scatter was added by forward-projecting a highly smoothed version of the images, which was added to the attenuated image sinogram scaled by the scatter fraction. Random counts were simulated by adding a uniform distribution to the true plus scatter count distribution scaled by the random fraction. We call the summation of $Tc$, $Sc$, and $Rc$ the total counts and denote it by $TC$.

Next we provide the figure-of-merits used for the comparisons: the Normalized Objective Function Value (NOFV), the Peak Signal-to-Noise Ratio (PSNR), the Normalized Relative Contrast (NRC), and the Central Line Profile (CLP). The NOFV is defined by
$$
\text{NOFV}(\bff^k):=\frac{\Phi(\bff^k)-\Phi_\text{ref}}{\Phi_0-\Phi_\text{ref}},
$$
where $\Phi$ denotes the objective function, $\Phi_0$ is the objective function value of the initial image, and $\Phi_\text{ref}$ denotes the reference objective function value. For simulation results, we set $\Phi_\text{ref}$ to the objective function value of the image reconstructed by 1000 iterations of ROS-PKMA \cite{lin2019krasnoselskii}. The PSNR is defined by
$$
\text{PSNR}(\bmf^k):=10\cdot\log_{10}\left(\frac{\text{MAX}^2}{\text{MSE}(\bmf^k)}\right),
$$
where MAX is the maximum possible pixel value of the image, MSE is the Mean Squared Error given by $\text{MSE}(\bmf^k):=\|\bff^k-\bff_{\text{true}}\|_2^2/d$, $\bff_{\text{true}}\in\bRd$ is the ground truth. This metric evaluates the quality of a reconstructed image compared to the ground truth. For the definition of NRC, we consider the following two regions: $\text{ROI}_H$ and $\text{ROI}_B$, which represent a region of interest within a specific hot sphere, and a background region positioned at a considerable distance from any hot sphere, respectively. The size of $\text{ROI}_B$ is identical to that of $\text{ROI}_H$, ensuring a fair comparison. The relative contrast (RC) is defined by $\text{RC}:=\frac{|E_{\text{ROI}_H}-E_{\text{ROI}_B}|}{E_{\text{ROI}_B}}$, where $E_{\text{ROI}_H}$ and $E_{\text{ROI}_B}$ represent the average activities of $\text{ROI}_H$ and $\text{ROI}_B$, respectively. The NRC is then given by
\[
\text{NRC}(\bff^k): =\frac{\text{RC}_{\bff^k}}{\text{RC}_{\text{true}}},
\]
where $\text{RC}_{\bff^k}$ and $\text{RC}_{\text{true}}$ are the relative contrast of $\bff^k$ and the ground truth respectively. This metric empowers us to evaluate the fluctuations in average activities and compare the relative contrast between reconstructed images and the ground truth image. The CLP is a metric used to analyze the intensity distribution along a central line within an image. By comparing the CLP of a reconstructed image to that of a ground truth image, we can assess the fidelity of the reconstruction and the extent to which important details have been preserved.

\begin{figure}[htbp]
\centering
\begin{tabular}{cc}
  \includegraphics[width=0.18\textwidth]{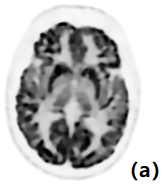}
  &\hspace{2em}\includegraphics[width=0.22\textwidth]{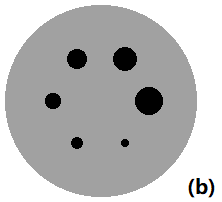}
\end{tabular}
\caption{(a) Brain phantom: high quality clinical PET brain image; (b) uniform phantom: uniform background with six uniform hot spheres of distinct radii.}
\label{fig:phantom}
\end{figure}

Our simulations utilized two 256$\times$256 numerical 2D phantom as shown in Figure \ref{fig:phantom}. The brain phantom was obtained from a high quality clinical PET brain image. The uniform phantom consists of a uniform background with six uniform hot spheres with distinct radii 4, 6, 8, 10, 12, 14 pixels. The activity ratio of the uniform hot spheres to the uniform background is 4:1. For the simulation experiments, we reconstructed the brain phantom to compare the NOFV and PSNR. Additionally, we reconstructed the uniform phantom to evaluate the NRC and CLP.  All simulations were performed in a 64-bit windows 11 laptop with Intel Core i7-13700H Processor at 2.40 GHz, 16 GB of DDR5 memory and 512GB NVMe SSD.

In all the experiments of APPGA, the preconditioning matrix $\bmP$ is set to the existing EM preconditioner \cite{krol2012preconditioned}, which is defined by
\begin{equation}\label{EM}
\bmP_{\EM}:=\beta\cdot\diag{\left(\frac{\fk}{\bmLambda}\right)},\ \ \beta>0.
\end{equation}
The vector $\bmLambda\in\bbRd$ is given by
$$\Lambda_j=\begin{cases}
(\bmA^\top{\bf1}_m)_j,&\text{if}\ (\bmA^\top{\bf1}_m)_j>0,\\
1,&\text{otherwise},
\end{cases}
$$
for $j=1,2,\ldots,d$. In the context of convergence analysis, it is essential that the preconditioner maintains consistency across all iteration steps. Nevertheless, in practical applications, updating the preconditioner during the initial iterations and subsequently fixing it in later iterations is permissible and does not undermine the convergence.

We show in Table \ref{tab:paramSHOTV} the setting of the parameters in the SHOTV regularized reconstruction model for the simulation experements. For the reconstruction of the brain phantom, to suppress the staircase artifacts and avoid excessive image smoothing, we empirically found that setting $\lambda_2 = \lambda_1$ was a reasonable choice. This approach also simplified the search for the optimal regularization parameter based on the maximal PSNR criterion. For the uniform phantom, due to its piecewise constant nature, the second-order TV regularization parameter $\lambda_2$ was set to 0. The determination of the setting of model and algorithmic parameters in APPGA were guided by the assessment of image quality and performance of the four figure-of-merits. Based on the convergence conditions of APPGA, the parameters $a$ and $b$ in the momentum step were set to $1/8$ and $1$, respectively. For PPGA, its parameter settings were consistent with those of APPGA. For PKMA, its algorithmic parameters were determined according to \cite{lin2019krasnoselskii}. To better evaluate the convergence speed of APPGA with different power parameters in the GN momentum, we shall show the comparison of APPGA with four different settings of $\omega$: $1/4$, $1/2$, $3/4$ and $1$. Let ACTc, NPFOV, and NPA represent the total attenuation corrected true counts, the number of pixels within the field of view, and the number of projection angles, respectively. For all the compared algorithms in the simulation experiments, we always set the initial images as the uniform disk $\text{TMC}\cdot{\bm1}_{\text{disk}}$ with the same size as the FOV, where $\text{TMC}:=\text{ACTc}/(\text{NPFOV}\cdot\text{NPA})$, ${\bm1}_{\text{disk}}$ is the image whose values are 1 within the disk and are 0 outside the disk \cite{lin2019krasnoselskii}.

\begin{table}[htbp]
\centering
\caption{Model parameters of the SHOTV regularized reconstruction model for the simulation experiments.}
\small
\renewcommand{\arraystretch}{1}
\begin{tabular}{c|c}
\hline
{Brain}&$\epsilon=0.001$ \\
{phantom}&$\lambda_1=\lambda_2=0.04$\\
\cline{1-2}
{Uniform}&$\epsilon=0.001$ \\
{phantom}&$\lambda_1=0.4$, $\lambda_2=0$ \\
\hline
\end{tabular}
\label{tab:paramSHOTV}
\end{table}

\subsection{Simulation results for APPGA}
In this subsection, we show the performance of APPGA for SHOTV regularized reconstruction model, and compare it with PPGA and PKMA. We set the total count $TC=6.8\times10^6$. For the step-size parameter $\beta$ in APPGA, we set $\beta=1$ for the brain phantom, and $\beta=0.1$ for the uniform phantom. First, we evaluate the performance of the compared algorithms in reconstructing the brain phantom. Since the time required for each iteration of the compared algorithms is very close, we can compare the convergence rates of these algorithms by observing the variation of the considered figure-of-merits with the number of iterations.

In Figure \ref{nonOS_BrainNOFVPSNR}, we present the plots of NOFV and PSNR versus the number of iterations. It is evident that APPGA converges at a progressively faster rate as the value of $\omega$ increases. For the four different settings of $\omega$ we considered, APPGA consistently converges much faster than PPGA. In the performance evaluation of NOFV, APPGA surpasses PKMA after 15 iterations when $\omega=1$, while this superiority is achieved after 31 iterations when $\omega=3/4$. Similarly, from the plot of PSNR, it is observe that APPGA outperforms PKMA after 17 iterations when $\omega=1$, and after 32 iterations when $\omega=3/4$. Figure \ref{nonOS_BrainReconIM} shows the reconstructed images of the brain phantom by using PKMA, PPGA and APPGA ($\omega=1$) at 25, 50 and 100 iterations, respectively. At these three different number of iterations, the quality of the images reconstructed by APPGA with $\omega=1$ is always the best. The image quality obtained by APPGA ($\omega=1$) at 25 (50) iterations is close to that of PKMA at 50 (100) iterations.

\begin{figure}[htbp]
\centering
\begin{tabular}{cc}
\includegraphics[width=0.48\textwidth]{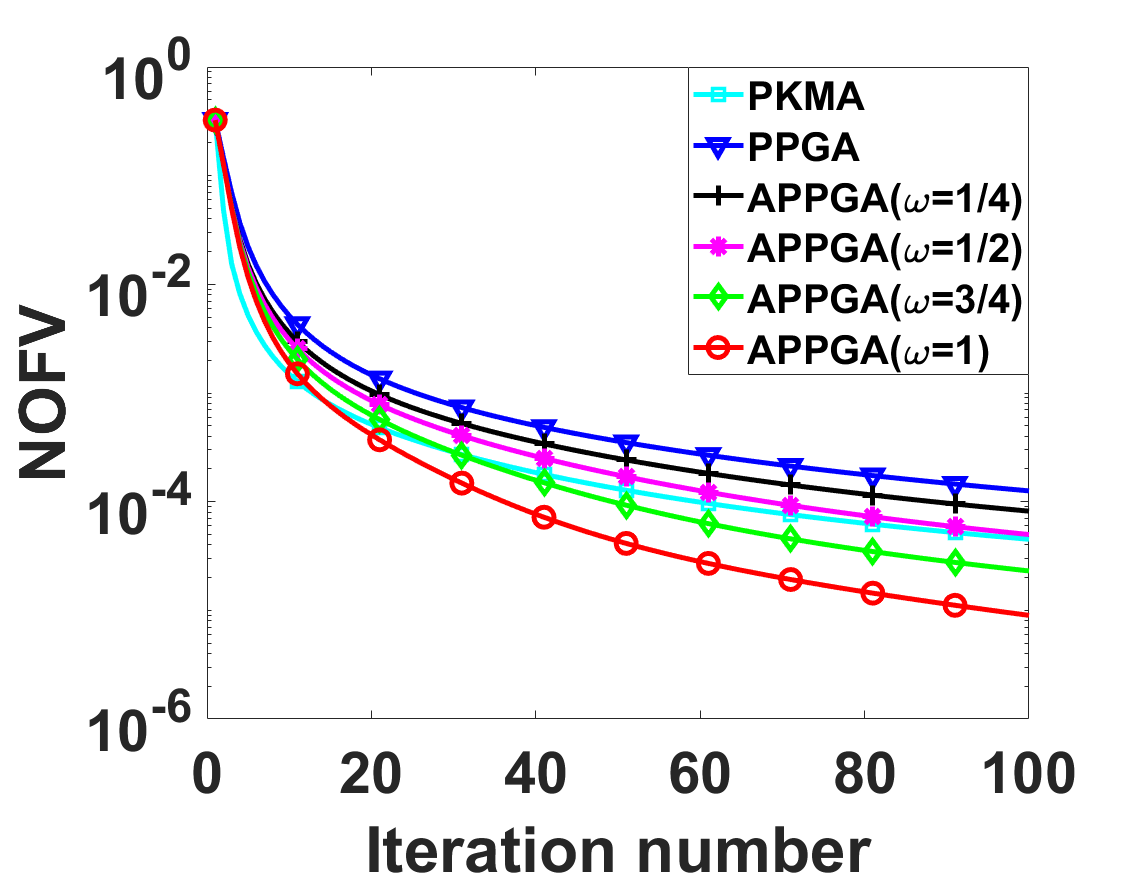}& \includegraphics[width=0.48\textwidth]{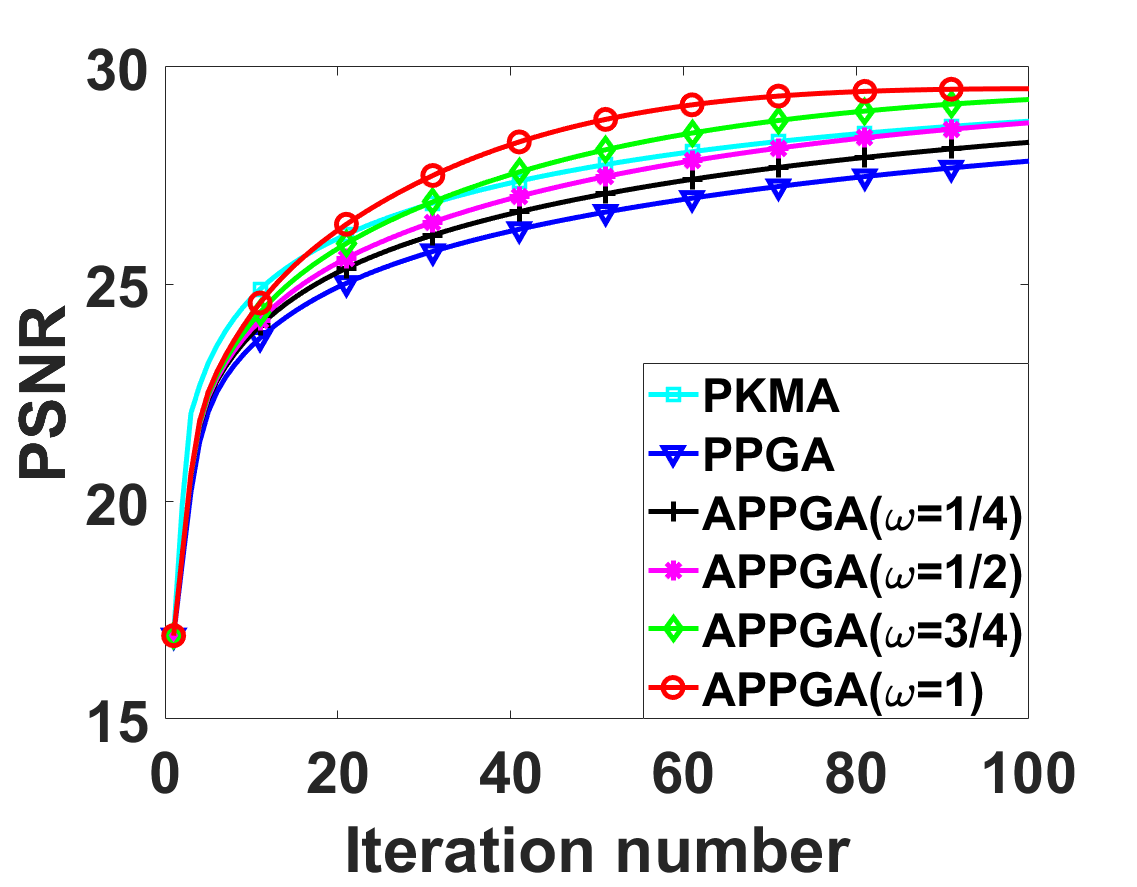}
\end{tabular}
\caption{NOFV (left) and PSNR (right) versus iteration number by PKMA, PPGA and APPGA ($\omega=1/4,1/2,3/4,1$).}
\label{nonOS_BrainNOFVPSNR}
\end{figure}

\begin{figure}[htbp]
\centering
\begin{tabular}{ccccc}
\small{PKMA} &\small{PPGA} &\small{APPGA($\omega$ =1)} \\
\vspace{-4pt}
\includegraphics[width=0.18\textwidth]{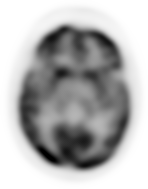}&
\includegraphics[width=0.18\textwidth]{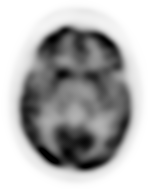}&
\includegraphics[width=0.18\textwidth]{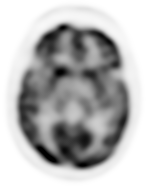}\\
\vspace{-4pt}
\includegraphics[width=0.18\textwidth]{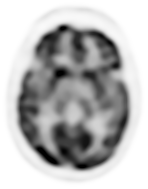}&
\includegraphics[width=0.18\textwidth]{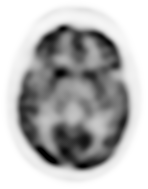}&
\includegraphics[width=0.18\textwidth]{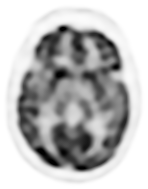}\\
\vspace{-4pt}
\includegraphics[width=0.18\textwidth]{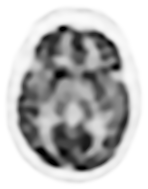}&
\includegraphics[width=0.18\textwidth]{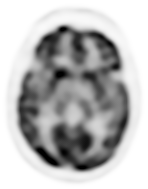}&
\includegraphics[width=0.18\textwidth]{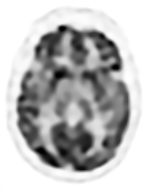}
\end{tabular}
\caption{Reconstructed brain images of PKMA, PPGA and APPGA($\omega$ =1): top to bottom rows are reconstructed by 25, 50 and 100 iterations, respectively.}
\label{nonOS_BrainReconIM}
\end{figure}

\newpage
Next, we analyzed the performance of the compared algorithms in reconstructing the uniform phantom. Figure \ref{nonOS_UniformNRC} presents the plot of NRC versus iteration number for the largest and smallest hot spheres of the uniform phantom. Figure \ref{nonOS_UniformCLP} shows the CLP of the reconstructed images at 50 and 100 iterations. Additionally, Figure \ref{nonOS_UniformReconIM} displays the reconstructed images of the uniform phantom. The results depicted in these figures demonstrate that APPGA outperforms PKMA and PPGA in reconstructing the uniform phantom. In the same way, larger $\omega$ in APPGA yields faster convergence.
\begin{figure}[htbp]
\centering
\begin{tabular}{cc}
\includegraphics[width=0.48\textwidth]{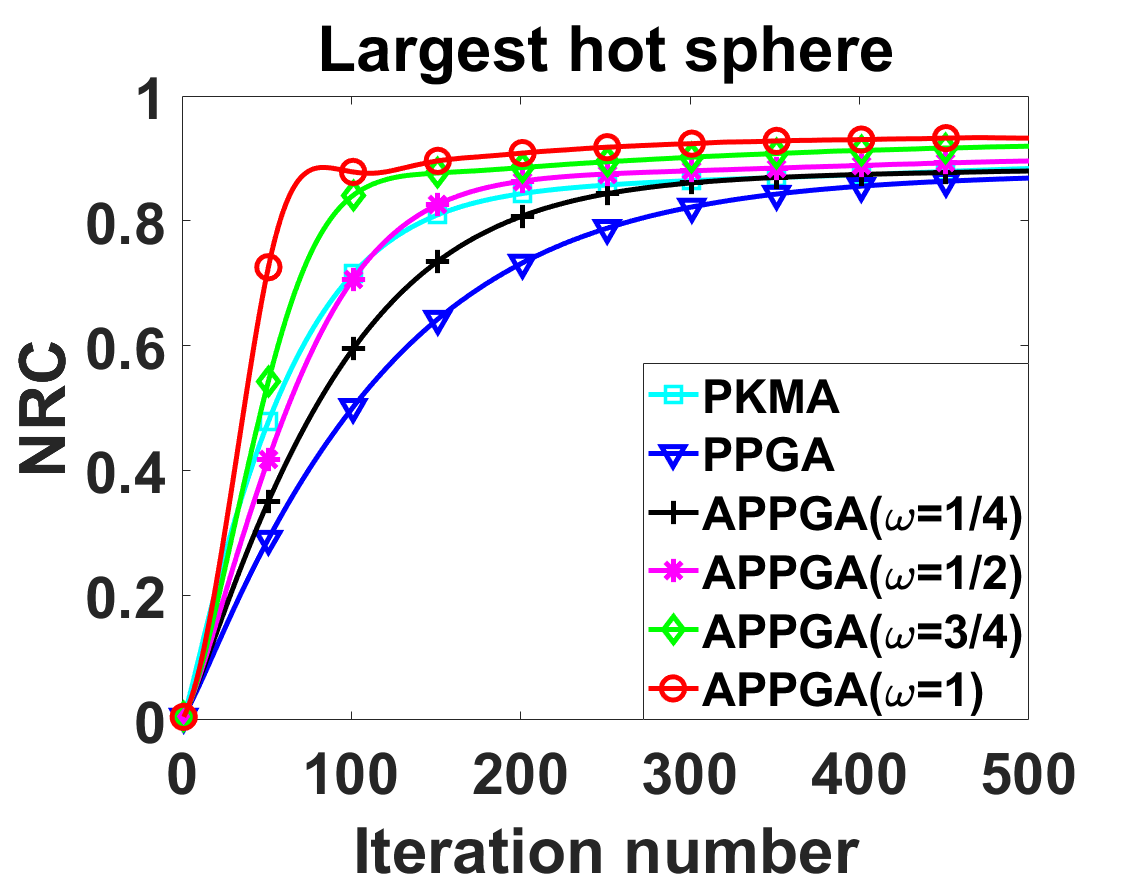}
&\includegraphics[width=0.48\textwidth]{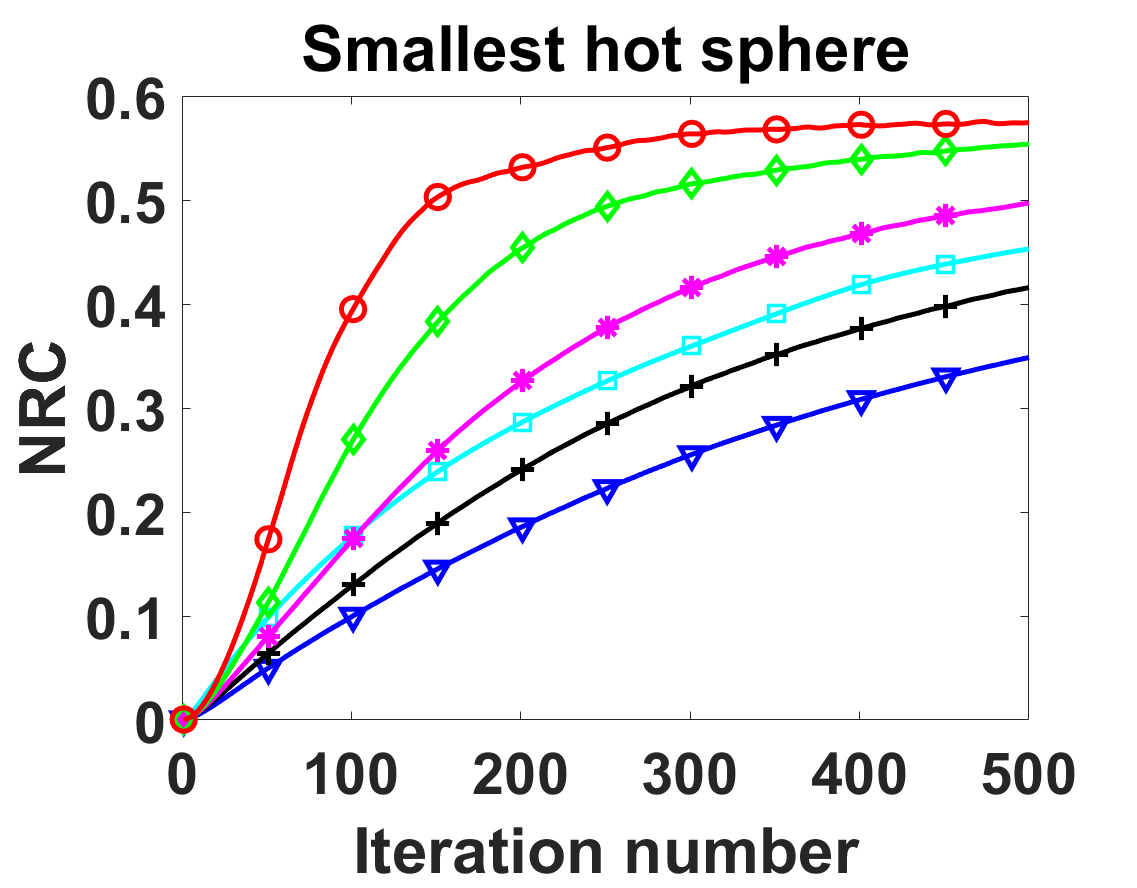}
\end{tabular}
\caption{NRC of the largest hot sphere (left) and the smallest hot sphere (right) versus iteration number by PKMA, PPGA and APPGA ($\omega=1/4,1/2,3/4,1$).}
\label{nonOS_UniformNRC}
\end{figure}

\begin{figure}[htbp]
\centering
\begin{tabular}{cc}
\includegraphics[width=0.48\textwidth]{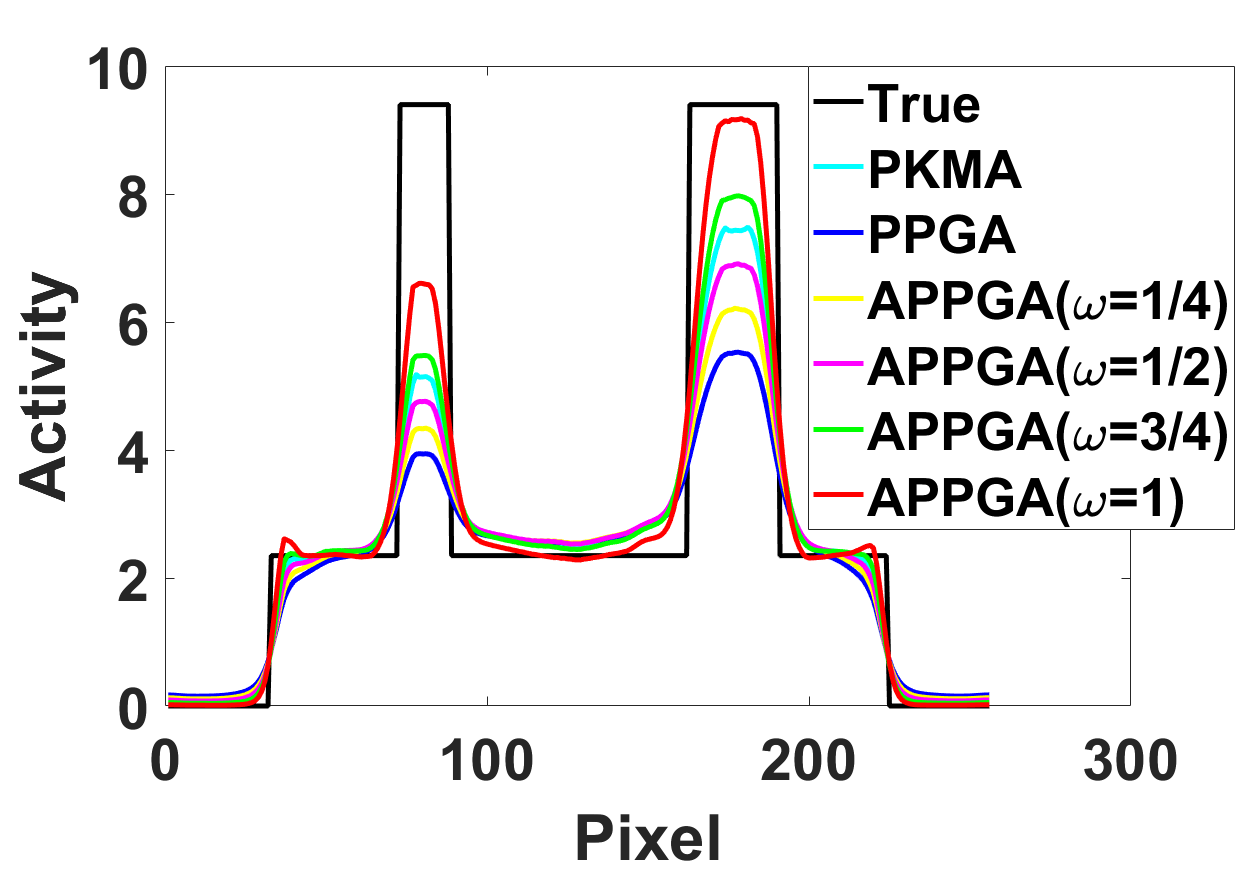}&
\includegraphics[width=0.48\textwidth]{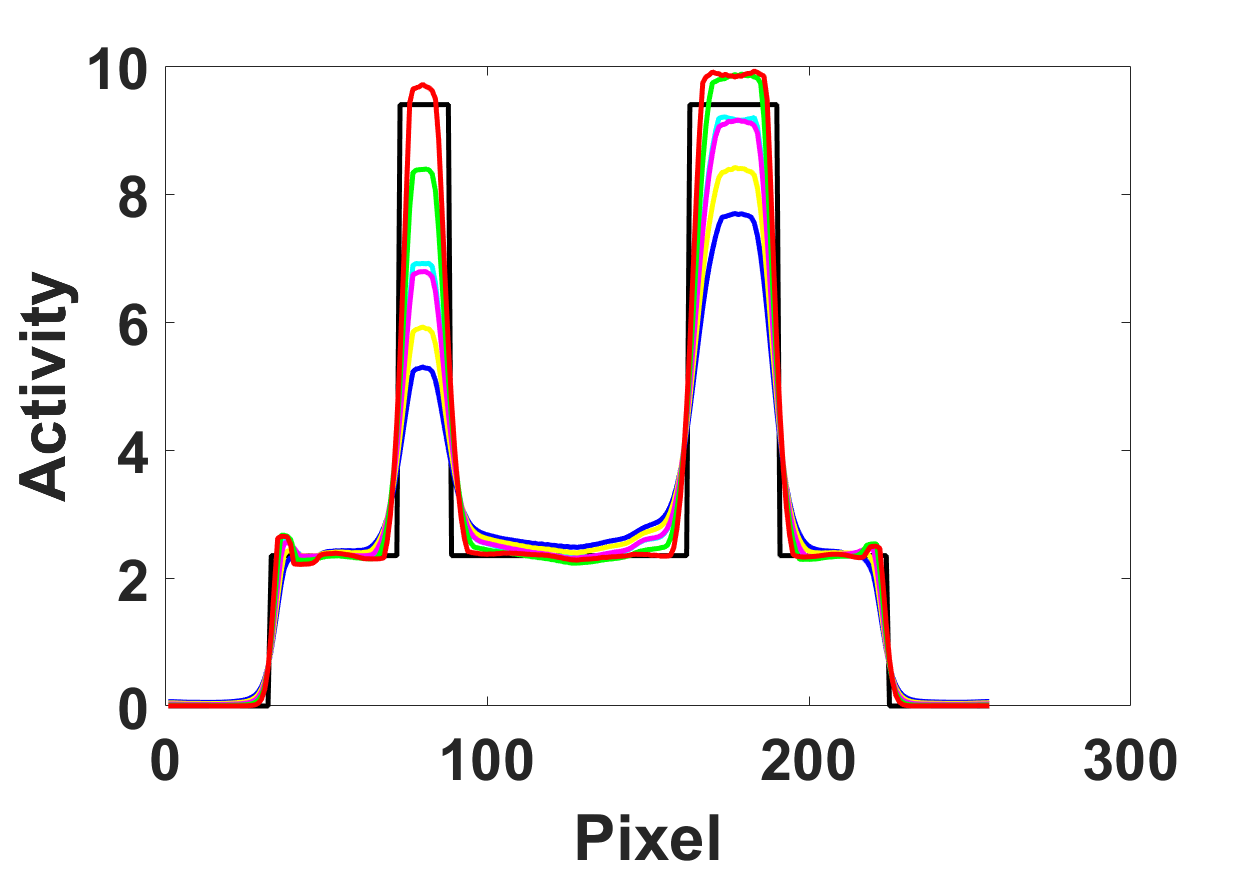}
\end{tabular}
\caption{CLP of the images reconstructed by 50 iterations (left) and 100 iterations (right) of PKMA, PPGA and APPGA ($\omega=1/4,1/2,3/4,1$).}
\label{nonOS_UniformCLP}
\end{figure}

\begin{figure}[htbp]
\centering
\begin{tabular}{ccc}
\small{PKMA} &\small{PPGA} &\small{APPGA($\omega$ =1)} \\
\vspace{-4pt}
\includegraphics[width=0.22\textwidth]{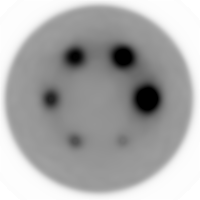}&
\includegraphics[width=0.22\textwidth]{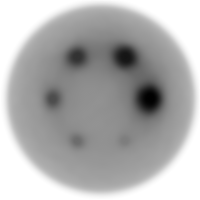}&
\includegraphics[width=0.22\textwidth]{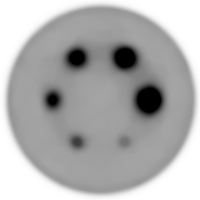}\\
\vspace{-4pt}
\includegraphics[width=0.22\textwidth]{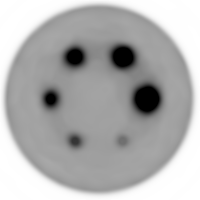}&
\includegraphics[width=0.22\textwidth]{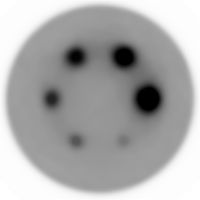}&
\includegraphics[width=0.22\textwidth]{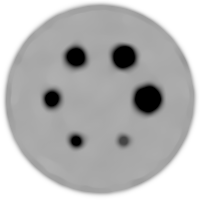}\\
\vspace{-4pt}
\includegraphics[width=0.22\textwidth]{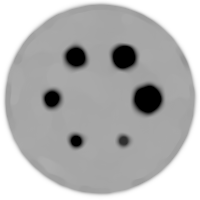}&
\includegraphics[width=0.22\textwidth]{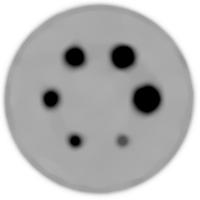}&
\includegraphics[width=0.22\textwidth]{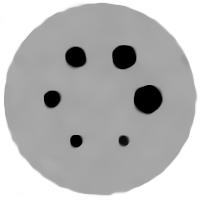}
\end{tabular}
\caption{Images of the uniform phantom reconstructed utilizing PKMA, PPGA and APPGA ($\omega=1$): top to bottom rows are reconstructed by 50, 100 and 500 iterations respectively.}
\label{nonOS_UniformReconIM}
\end{figure}

\newpage
\subsection{Extension of the generalized Nesterov momentum}
In Section \ref{sec:convAPPGA}, we establish that the APPGA with the Generalized Nesterov (GN) momentum exhibits the convergence rate presented in Theorem \ref{thm:mainconvrate} when solving a two-term optimization model, where one term is convex and differentiable with a Lipschitz continuous gradient, and the other is convex with a closed-form proximity operator. Notably, though the GN momentum technique may not guarantee the same convergence or convergence rate results for more complex models, its generality offers greater potential for applying momentum acceleration techniques to more intricate problems compared to the conventional Nesterov momentum.

In this subsection, we evaluate the acceleration effect of the GN momentum technique on the FPPA algorithm for solving the original HOITV regularized PET image reconstruction model \eqref{model:HOTV}, which includes multiple nondifferentiable terms. We demonstrate through numerical experiments that for such a complex objective function, conventional Nesterov momentum may fail to guarantee convergence. In contrast, the GN momentum scheme exhibits more robust and adaptable behavior when we decrease the exponent parameter $\omega$ in \eqref{GNtheta}.

We recall from \cite{li2015multi,li2016fast,lin2019krasnoselskii} that the iterative scheme of Fixed-Point Proximity Algorithm (FPPA) for solving model \eqref{model:HOTV} is given by
\begin{equation*}\label{FPPAiter}
\begin{cases}
\bmf^{k+1}=\max\left(\fk-\bmP\left(\nabla F(\fk)+\bmB_1^\top\bk+\bmB_2^\top\ck\right),{\bm0}\right)\\
\bmb^{k+1}=\rho_1\Big(\mcI-\prox_{\frac{\lambda_1}{\rho_1}\varphi_1}\Big)\Big(\frac{1}{\rho_1}\bk+\bmB_1(2\bmf^{k+1}-\fk)\Big)\\
\bmc^{k+1}=\rho_2\Big(\mcI-\prox_{\frac{\lambda_2}{\rho_2}\varphi_2}\Big)\Big(\frac{1}{\rho_2}\ck+\bmB_2(2\bmf^{k+1}-\fk)\Big)\\
\end{cases}.
\end{equation*}
Then the iterative scheme of Accelerated FPPA (AFPPA) is given by
\begin{equation*}\label{AFPPAiter}
\begin{cases}
\tfk=\bmf^k+\theta_k(\bmf^k-\bmf^{k-1})\\
\tbk=\bmb^k+\theta_k(\bmb^k-\bmb^{k-1})\\
\tck=\bmc^k+\theta_k(\bmc^k-\bmc^{k-1})\\
\bmf^{k+1}=\max\left(\tfk-\bmP\left(\nabla F(\tfk)+\bmB_1^\top\tbk+\bmB_2^\top\tck\right),{\bm0}\right)\\
\bmb^{k+1}=\rho_1\Big(\mcI-\prox_{\frac{\lambda_1}{\rho_1}\varphi_1}\Big)\Big(\frac{1}{\rho_1}\tbk+\bmB_1(2\bmf^{k+1}-\tfk)\Big)\\
\bmc^{k+1}=\rho_2\Big(\mcI-\prox_{\frac{\lambda_2}{\rho_2}\varphi_2}\Big)\Big(\frac{1}{\rho_2}\tck+\bmB_2(2\bmf^{k+1}-\tfk)\Big)\\
\end{cases}.
\end{equation*}
The preconditioning matrix $\bmP$ is set as the EM preconditioner given by \eqref{EM}. We refer to the AFPPA with the GN momentum in \eqref{GNtheta} as AFPPA-GN, and the AFPPA with the original Nesterov momentum \cite{nesterov1983method,beck2009fast} as AFPPA-N, where the Nesterov momentum parameters are given by
$$
\theta_k=\frac{t_{k-1}-1}{t_k},\ \ t_k=\frac{1+\sqrt{1+4t_{k-1}^2}}{2},\ \ k\in\bbN_+.
$$

We next describe how we set up the parameters of the experiment for the comparison of AFPPA-GN with AFPPA-N.  We use these two algorithms to reconstruct the brain phantom shown in Figure \ref{fig:phantom} (a). In this experiment, we set $TC=1.7\times10^7$ and choose the regularization parameters $\lambda_1=\lambda_2=0.007$. The other two algorithmic parameters $\rho_1$ and $\rho_2$ are set to $\rho_1=\frac{1}{2\times8\times p_{\max}}$ and $\rho_2=\frac{1}{2\times64\times p_{\max}}$ according to \cite{lin2019krasnoselskii}, where $p_{\max}$ denotes the maximal diagonal entry of the preconditioning matrix $\bmP$. To better evaluate the convergence of the algorithms, we define the Relative Error (RE) at the $k$th iteration by
$$
\text{RE}(k):=\frac{\|\bmf^k-\bmf^{k-1}\|_2}{\|\bmf^k\|_2}.
$$
A steadily diminishing RE indicates the progressive convergence of the algorithm. Figure \ref{nonOS_BrainAFPPA} shows the plots of lg(RE) and PSNR against the number of iterations.
\begin{figure}[htbp]
\centering
\begin{tabular}{cc}
\includegraphics[width=0.4\textwidth]{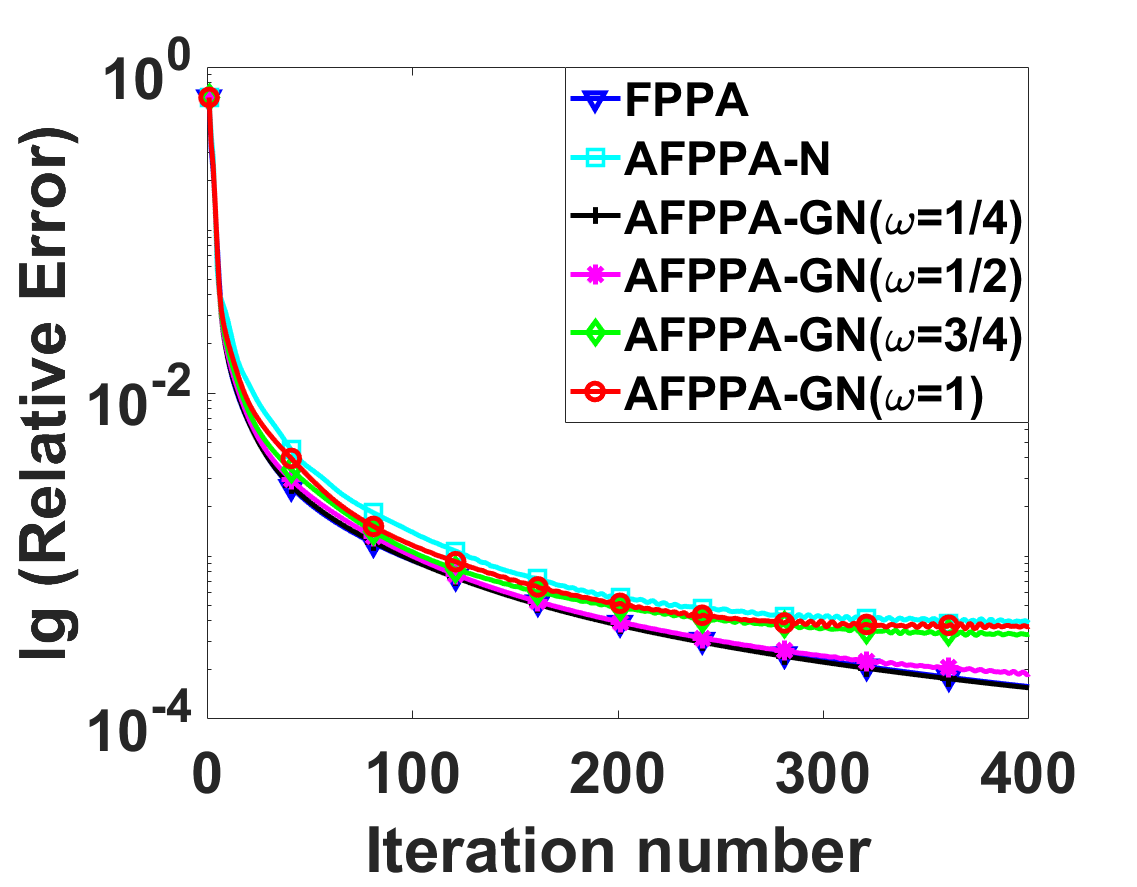}
&\includegraphics[width=0.4\textwidth]{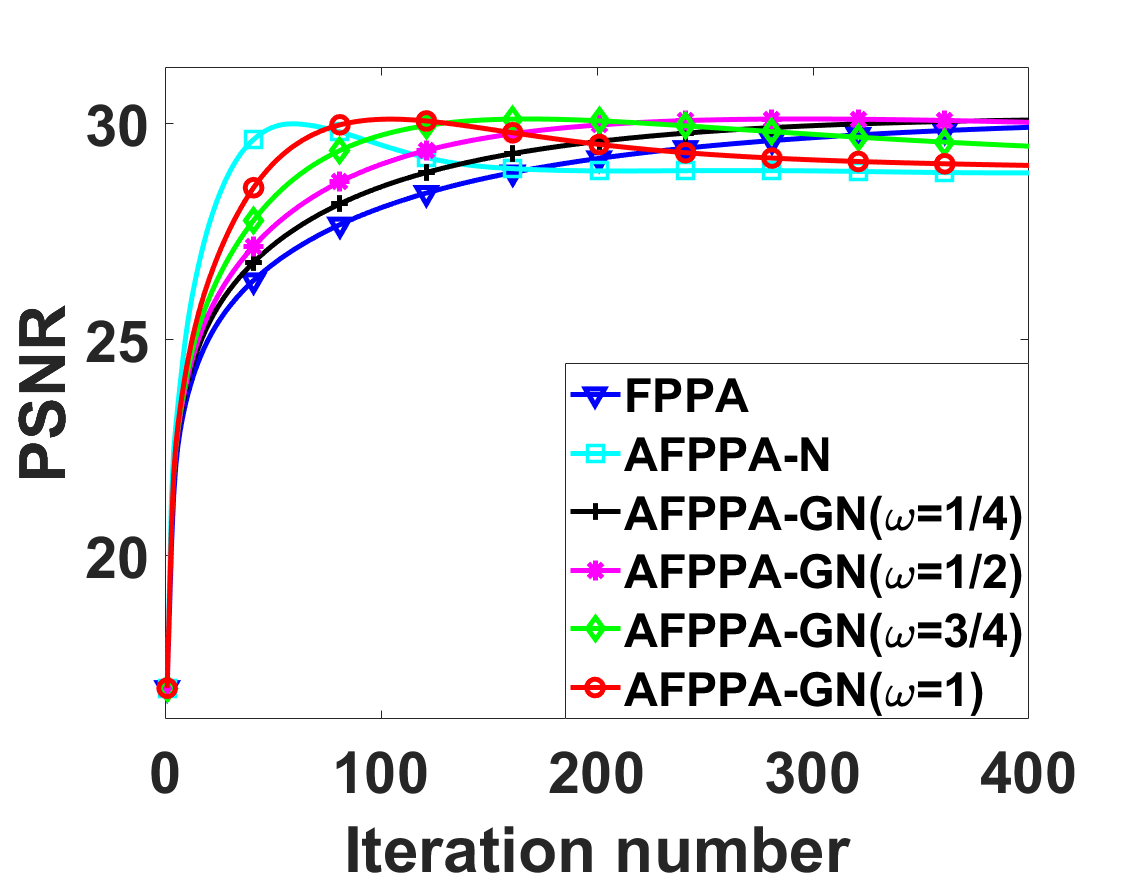}
\end{tabular}
\caption{lg(RE) (left) and PSNR (right) versus iteration number by FPPA, AFPPA-N and AFPPA-GN($\omega=1/4,1/2,3/4,1$).}
\label{nonOS_BrainAFPPA}
\end{figure}

From Figure \ref{nonOS_BrainAFPPA}, we observe that for such a complex objective function, increasing the exponent parameter leads to less stable convergence of the algorithm, and may fail to converge. When the exponent parameter $\omega$ is set to 1/4 or 1/2, the RE of AFPPA-GN continues to decline over 400 iterations, while its PSNR stabilizes at a higher level than FPPA. This indicates that with an appropriate exponent parameter, AFPPA-GN not only ensures effective convergence but also demonstrates superior convergence speed compared to FPPA. In contrast, in the later stages of iteration, the RE of AFPPA-N ceases to decrease, indicating that the algorithm fails to converge. This lack of convergence also causes the PSNR of AFPPA-N to decline after initially increasing to a certain level. Consequently, the final PSNR value is much lower than that of FPPA without momentum acceleration and AFPPA-GN with exponent parameters set to 1/2 or 1/4. These experimental results demonstrate that the GN momentum technique has broader applicability compared to conventional Nesterov momentum. By controlling the exponent parameter $\omega$, it not only maintains acceleration performance but also ensures stable convergence.

\section{Concluding remarks}
This paper proposes a readily implemented and mathematically robust Accelerated Preconditioned Proximal Gradient Algorithm (APPGA) to address a class of PET image reconstruction models with differentiable regularizers. We prove that APPGA with the Generalized Nesterov (GN) momentum enjoys favorable theoretical convergence properties. Specifically, we establish that it achieves an $o(\frac{1}{k^{2\omega}})$ convergence rate in terms of the function value, and an $o(\frac{1}{k^{\omega}})$ convergence rate in terms of the distance between consecutive iterates, where $\omega\in(0,1]$ is the power parameter of the GN momentum. To improve the efficiency and convergence rate for nondifferentiable higher-order Isotropic Total Variation (ITV) regularized PET image reconstruction, we smooth the ITV term in the model and then apply the APPGA to solve it. Numerical results reveal that as $\omega\in(0,1]$ increases, APPGA exhibits a significantly faster convergence rate, surpassing the performance of both PKMA and PPGA in the reconstruction of PET images using the Smoothed Higher-Order ITV regularization (SHOITV). Moreover, we provide insights into the extensive applicability of the GN momentum in dealing with the original nonsmoothed HOITV regularized PET model, emphasizing its robustness and adaptability when compared to the conventional Nesterov's momentum. We conclude that the proposed method based on APPGA with the GN momentum and the SHOITV regularization performs favorably and is very promising for PET image reconstruction.

\section*{Acknowledgments}
Yizun Lin is supported in part by the National Natural Science Foundation of China under Grant 12401120, by Guangdong Basic and Applied Basic Research Foundation under Grant 2021A1515110541, and by the Science and Technology Planning Project of Guangzhou under Grant 2024A04J3940. C. Ross Schmidtlein is supported in part by the NIH/NCI R21 CA263876, NIH/NIBIB R01 EB034742, and by the MSK Cancer Center Support Grant\slash Core Grant P30 CA008748. Deren Han is supported in part by the Ministry of Science and Technology of China (No.2021YFA1003600), by the R\&D project of Pazhou Lab (Huangpu) (2023K0604), and by the National Natural Science Foundation of China under Grants 12131004 and 12126603.

\section*{References}
\bibliographystyle{acm}
\bibliography{bibfile}

\end{document}